\documentclass[11pt]{article}

\usepackage{amsmath,amsfonts,amsthm,amssymb,blkarray}
\usepackage{mathrsfs, graphicx,color,supertabular}
\usepackage{ulem}

\newtheorem{thm}{Theorem}[section]
\newtheorem{lem}[thm]{Lemma}
\newtheorem{cor}[thm]{Corollary}
\newtheorem{re}[thm]{Remark}

\newtheorem{prop}[thm]{Proposition}

\newtheorem{prob}{Problem}
\newtheorem*{claim}{Claim}

\newcommand{\diam}{{\rm diam}}

\title{\bf \large On graphs with exactly one anti-adjacency eigenvalue and beyond\footnote{The first two authors are supported by National Natural Science Foundation of China (No. 11971274) and the third author is supported by National Natural Science Foundation of China (No. 11771247).}}

\author{
{\small Jianfeng Wang$^{a,}$\footnote{Corresponding author.
\newline{\it \hspace*{5mm}Email addresses:} jfwang@sdut.edu.cn (J.F. Wang),
xyleiyuki@aliyun.com (X. Lei), lumei@tsinghua.edu.cn (M. Lu), srgnrzs@gmail.com (S. Sorgun), hakankucuk1979@gmail.com (H. K\"{u}\c{c}\"{u}k).}\;, \ \ Xingyu Lei$^{a}$, \ \ Mei Lu$^{b}$, \ \  Sezer Sorgun$^{c}$, \ \ Hakan K\"{u}\c{c}\"{u}k$^{c}$}\\[2mm]
\footnotesize $^a$School of Mathematics and Statistics, Shandong University of Technology, Zibo 255049, China\\
\footnotesize $^b$Department of Mathematical Sciences, TsingHua University, Beijing 100084, China\\
\footnotesize $^c$Department of Mathematics, Nev\c{s}ehir Hac{\i} Bekta\c{s} Veli University, Turkey}
\date{ }
\begin{document}
\maketitle
\begin{abstract}
The anti-adjacency matrix of a graph is constructed from the distance matrix of a graph by keeping each row and each column only the largest distances. This matrix can be interpreted as the opposite of the adjacency matrix, which is instead constructed from the distance matrix of a graph by  keeping in  each row and each column only the distances equal to 1. The (anti-)adjacency eigenvalues of a graph are those of its (anti-)adjacency matrix. Employing
a novel technique introduced by Haemers [Spectral characterization of mixed extensions of small graphs, Discrete Math. 342 (2019) 2760--2764], we characterize all connected graphs with exactly one positive anti-adjacency eigenvalue, which is an analog of Smith's classical result that a connected graph with exactly one positive adjacency eigenvalue iff it is a complete multipartite graph. On this basis, we identify the connected graphs with all but at most two anti-adjacency eigenvalues equal to $-2$ and $0$. Moreover, for the anti-adjacency matrix we determine the HL-index of graphs with exactly one positive anti-adjacency eigenvalue, where the HL-index  measures how large in absolute value may be the median  eigenvalues of a graph. We finally propose some problems for further study.\\

\end{abstract}

\baselineskip=0.202in

\section{\large Introduction}

To study the graphs determined by the spectrum of  adjacency matrices, Haemers \cite{Haem} introduced a useful operation on a graph named as {\it mixed extension}. Consider a graph $G$ with vertex set $\{1,\ldots, n\}$. Let $V_1,\dots,V_n$ be mutually disjoint nonempty finite sets. We define a graph
$H$ with vertex set the union of $V_1,\ldots, V_n$ as follows. For each $i$, the vertices of $V_i$ are either all mutually adjacent ($V_i$ is a clique), or all mutually nonadjacent ($V_i$ is a coclique). When $i \neq j$, a vertex of $V_i$ is adjacent to a vertex of $V_j$ if and only if $i$ and
$j$ are adjacent in $G$. We call $H$ a {\it mixed extension} of $G$. We represent a mixed extension by an $n$-tuple $(t_1, \ldots, t_n)$ of nonzero integers, where $t_i > 0$ indicates that $V_i$ is a clique of order $t_i$, and $t_i < 0$ means that $V_i$ is a coclique of order $-t_i$. Haemers's definition is more convenient and powerful, which is motivated by a concrete question for which the pineapple graphs are determined by the adjacency  spectra \cite{top-pin2}. We refer to \cite{Haem,Haem1} for basic results on mixed extensions and to \cite{bro-hae} for graph spectra.

Along with other techniques, Haemers \cite{Haem} determined  all graphs with at most three eigenvalues unequal to 0 and $-1$, consisting of all mixed extensions of graphs on at most three vertices together with some particular mixed extensions of the paths $P_4$ and $P_5$. Subsequently, Haemers et al. \cite{Haem1} investigated the mixed extension of $P_3$ on being determined by the adjacency spectrum and presented several cospectral families. Comparatively, Cioab\v{a} et al. \cite{Cioa1} provided another method  for constructing graphs with all but two eigenvalues equal to $\pm 1$. Moreover, Cioab\v{a} et al. \cite{Cioa2}  identified the graphs with all but two eigenvalues equal to $-2$ or 0. For the  distance matrices of graphs, an approach has been successful for the graphs with exactly two distance eigenvalues different from $-1$ and $-2$ \cite{xueyi} or $-1$ and $-3$ \cite{Lulu}.

In this paper, one will see that the mixed extension of a graph has new applications to the so-called {\it anti-adjacency matrix} of graphs. We  only consider finite, simple and connected graphs. Let $G = (V(G), E(G))$ be a graph with order $|V(G)| =n$. The {\it distance} $d_G(v,w)$  between two vertices $v$ and $w$ is the minimum length of the paths joining them. The {\it diameter} of $G$, denoted by dim$(G)$, is the greatest distance between any two vertices in $G$. The {\it eccentricity} $\varepsilon_G(u)$ of the vertex $u \in V(G)$ is given by $\varepsilon_G(u)= \max\{d(u, v) \mid v \in V(G)\}$. Then
the {\it anti-adjacency matrix} (or {\it eccentricity matrix}) $\mathcal{A}(G)=(\epsilon_{uv})$ of $G$ are defined as follows \cite{wang-dam-anti}:
\begin{equation}\label{A-defi}
\epsilon_{uv}=\left\{
   \begin{array}{ll}
    d_G(u,v)  &\mbox{if}\phantom{o} d_G(u,v)= \min\{\varepsilon_G(u),\varepsilon_G(v)\},\\
     0        &\mbox{otherwise}.
    \end{array}
    \right.
\end{equation}
By comparing the definitions, it turns out that  $\mathcal{A}(G)$ is equal to the $D_{\rm MAX}$-matrix introduced by Randi\'c in \cite{Rand} as a tool for Chemical Graph Theory. Anyway, since the importance of vertex-eccentricity is not limited to applications to chemistry,  the author asserted that such matrix might open new directions of exploration in other branches of graph theory as well.

The matrix  $\mathcal{A}(G)$ is constructed from the distance matrix by only keeping the largest distances for each row and each column, whereas the remaining entries become null. That is why  $\mathcal{A}(G)$ can be interpreted as the opposite of the adjacency matrix,
which is instead constructed from the distance matrix by keeping only distances equal to $1$ on each row and each column.
From this point of view, $A(G)$ and $\mathcal{A}(G)$ are extremal among all possible
distance-like matrices.  As a contrast with $A(G)$, the anti-adjacency matrix has some fantastic properties, one of which is that $\mathcal{A}(G)$ of a  connected graph is not necessarily irreducible. See the published papers \cite{he,mah-gur-raj-aro,wang-li-CILS,wang-dam-anti,wang-dam-spec,wei-he-li} and the arXiv preprints \cite{mah-gur-raj-aro-1,tura} for more results about this newer matrix.

We next introduce some notations borrowed from spectral graph theory. The {\it $\mathcal{A}$-polynomial} of $G$ is defined as $\phi(G,\lambda) = \det(\lambda{I} - \mathcal{A}(G))$, where $I$ is the identity matrix. The roots of the $\mathcal{A}$-polynomial are the  {\it $\mathcal{A}$-eigenvalues} and the {\it
$\mathcal{A}$-spectrum}, denoted also by ${\rm Spec}_{\mathcal{A}}(G)$, of $G$ is  the
multiset consisting of the $\mathcal{A}$-eigenvalues. Since $\mathcal{A}(G)$ is symmetric, the $\mathcal{A}$-eigenvalues are real. Let $\xi_1 \geq
\xi_2 \geq \cdots \geq \xi_n$ be the anti-adjacency eigenvalues of a graph with order $n$. If $\xi_1' >
\xi_2' > \cdots > \xi_k'$ are all distinct $\mathcal{A}$-eigenvalues, then the
$\mathcal{A}$-spectrum can be written as
$$
{\rm Spec}_{_{\mathcal{A}}}(G) = \{\xi_1,\xi_2, \ldots, \xi_n\} =
 \left\{\!\!\begin{array}{cccc}
                                 \xi_1' & \xi_2' & \cdots & \xi_k' \\
                                 m_1   &  m_2           & \cdots &  m_k
                                \end{array}
                         \!\!\right\},
$$
where $m_i$ is the algebraic multiplicity of the eigenvalue $\xi_i'$ ($1 \leq i \leq k$).

For any  graph matrix $M$,  the {\it $M$-cospectral graphs} are non-isomorphic graphs with the same $M$-spectrum. We say that $G$ is {\it determined by the $M$-spectrum} if  no $M$-cospectral graphs of $G$ exist. Usually, $M$ is the adjacency, or the Laplacian, or the anti-adjacency matrices and so on.

The most important reason why the authors \cite{wang-lu-0} tended to build a spectral theory based on the anti-adjacency matrix is that they tried to detect the
proportion of cospectral graphs relating to two famous conjectures: One is that almost all graphs are adjacency cospectral posed by Schwenk \cite{sch-almost};
the other is that almost all graphs are determined by the adjacency (or Laplacian) spectrum formally proposed by Haemers \cite{hae-conj}. In their paper, the
authors \cite{wang-lu-0}  showed that, when $n \rightarrow \infty$, the fractions of non-isomorphic cospectral graphs with respect to the adjacency and the
anti-adjacency matrix behave like those only concerning the self-centered graphs with diameter two. Moreover,  they also obtained that the connected graphs have
exactly two distinct $\mathcal{A}$-eigenvalues iff they are $r$-antipodal graphs, which could be used to construct much more $\mathcal{A}$-cospectral graphs.

Recall, a classical result, due to Smith \cite[Theorem 6.7]{cve-book}, is that  a connected graph has exactly one positive adjacency eigenvalue if and only if $G$ is a complete multipartite graph $K_{n_1,n_2,\cdots,n_t}$, which is just the mixed extension  of complete graph $K_t$ of type $(-n_1, -n_2, \ldots, -n_t)$. Contrastively, we will determine the graphs with exactly one positive anti-adjacency eigenvalue, which is the mixed extension of star $K_{1,k+1}$ in Theorem \ref{E-main-1-one}. Based on this result, we classify the graphs with all but two anti-adjacency eigenvalues equal to $-2$ and $0$ in Theorem \ref{E-main-2-two}.
Relatively,  Cioab\u{a} et.al \cite{Cioa2} determined all connected graphs for which adjacency  matrices have at most two eigenvalues  not equal to $-2$ and $0$, which are different from the graphs in the above theorem. In  their classification, there are thirteen families of such graphs. Additionally, for the adjacency matrix, the multiplicity of eigenvalue 0 is well-known as the {\it nullity} of a graph that has been studied widely; while it was a hot topic to determined the graphs with the smallest eigenvalue at least $-2$, which has been summarized in \cite{cve-book-2}. For the anti-adjacency matrix, the graphs with the smallest $\mathcal{A}$-eigenvalue at least $-2$ was characterized in \cite{wang-li-CILS}.

Fowler and Pisanski [16,17] introduced the notion of the HL-index of a graph w.r.t. the adjacency matrix. It is related to the HOMO-LUMO separation studied in theoretical chemistry. Similarly,  the {\it {\rm HL}-index} $R_{\mathcal{A}}(G)$ w.r.t.  anti-adjacency matrix of a (molecular) graph $G$ of order $n$  is defined as $$R_{\mathcal{A}}(G) = \max\{|\xi_H|, |\xi_L|\},$$ where $H = \lfloor\frac{n+1}{2}\rfloor$, $L = \lceil\frac{n+1}{2}\rceil$. See \cite[eg.]{mohar-HL1,mohar-HL2,mohar-HL3} for more details about the  HL-index w.r.t. adjacency matrices of graphs. Actually, for the anti-adjacency matrix, by cumbersome calculations we can completely determine the HL-index of graphs with exactly one positive anti-adjacency eigenvalue.

In order to state the following main results in this paper, we describe the mixed extensions of a star. Let $K_{1,k+1}$  be a star with the vertex $v_0$ of degree $k+1$ and the other vertices $v_{1},v_{2},\ldots,v_{k+1}$.  We represent by  $S(t_0,t_1, \ldots, t_k)$ the mixed extension of the star $K_{1,k+1}$ of the type $(t_0, t_1, \ldots, t_k)$.
If $t_1 \geq t_2 \geq \cdots \geq t_{q} \geq 2 > t_{q+1} = \cdots = t_k = 1$, then $S(t_0,t_1, \ldots, t_k)\cong S(t_0, -p, t_1,\cdots, t_q)$, where $0 \leq  p, q \leq k$ $(p = k - q)$ and
$t_j \geq 2$ $(1 \leq j \leq q)$. It is worth noting that each $t_i\geq 2$ ($i=1,2,\cdots, k$) if  $p=0$.

\begin{thm}\label{E-main-1-one}
A connected graph $G$ has exactly one  positive $\mathcal{A}$-eigenvalue if and only if $G$ is the mixed extension of star $S_{1,q+1}$ of type $(t_0, -p, t_1, \ldots, t_q)$ with $p,q\geq 0$ and $t_j \geq 2$ ($1 \leq j \leq q$), where
\begin{itemize}
\item[$\mathrm{(i)}$]
$t_0 = 1$, $p +q \geq 1$;
\item[$\mathrm{(ii)}$]
$t_0 = 2$, $p, q \geq 0$;
\item[$\mathrm{(iii)}$]
$t_0 = 3$, $0 \leq q \leq 4 - p$ $(0 \leq p \leq 4)$;
\item[$\mathrm{(vi)}$]
$t_0 = 4$, $0 \leq q \leq 3 - p$ $(0 \leq p \leq 3)$;
\item[$\mathrm{(v)}$]
$t_0 \geq 5$, $0 \leq q \leq 2 - p$ $(0 \leq p \leq 2)$.
\end{itemize}
\end{thm}

\begin{thm}\label{E-main-2-two}
Let $G$ be a connected graph. Then
\begin{itemize}
\item[{\rm (i)}]
 No graph has exactly one $\mathcal{A}$-eigenvalue different from $0$ and $-2$.
\item[{\rm (ii)}]
The graph with all but two $\mathcal{A}$-eigenvalues equal to $-2$ and $0$ if and only if $G$ is the mixed extension of star $S_{1,1}$ of type $(-t_0, -t_1)$, where $t_0, t_1 \geq 1$.
\end{itemize}
\end{thm}

Write the set $\{t_1, t_2,  \cdots, t_q\}$ as the multiset  $\{k_1 \cdot t_1, \cdots, k_h \cdot t_h\}$, where $k_i$ is the number of $t_i's$ ($1\leq i \leq h$). Clearly,  $S(t_0, -p, t_1, \ldots, t_q) \cong S(t_0,-p,k_1 \cdot t_1, \cdots, k_h \cdot t_h)$ and $\sum\limits_{i=1}^{h}k_i= q$.

\begin{thm}\label{E-HL-index}
Let $G = S(t_0,-p,k_1 \cdot t_1, \cdots, k_h \cdot t_h)$ be the mixed extension of star $K_{1,q+1}$ defined in Theorem \ref{E-main-1-one}.  For $p+q \leq 1$, then $R_{\mathcal{A}}(G) =1$. For $p+q \geq 2$,
\begin{itemize}
\item[$\mathrm{(i)}$]
if $p+q \leq \lceil \frac{n-2t_0}{2} \rceil$, then $R_{\mathcal{A}}(G) =0$.
\item[$\mathrm{(ii)}$]
if $p+q = \lceil \frac{n-2t_0+2}{2} \rceil$,
for $t_0=1, q=0$ or $t_0=3, p+q=4$ or $t_0=4, p+q =3$, $R_{\mathcal{A}}(G)= 0$; for the other cases in conditions (i)-(v) of Theorem \ref{E-main-1-one},  $R_{\mathcal{A}}(G) \in (0,1)$.
\item[$\mathrm{(iii)}$]
if $\lceil \frac{n-2t_0+4}{2}\rceil \leq p+q \leq \lceil \frac{n}{2} \rceil$, then $R_{\mathcal{A}}(G)=1$.
\item[$\mathrm{(iv)}$]
if $p+q \geq \lceil \frac{n+2}{2} \rceil$ and $q\leq \lceil \frac{n-2}{2} \rceil$, then $R_{\mathcal{A}}(G)=2$.
\item[$\mathrm{(v)}$]
if $p+q \geq \lceil \frac{n+2}{2} \rceil$ and
\begin{itemize}
\item[$\mathrm{(a)}$]
for $q= \lceil \frac{n}{2} \rceil$, then  $R_{\mathcal{A}}(G)\in(-2, -2t_h)$;
\item[$\mathrm{(b)}$]
for $\lceil \frac{n+2}{2} \rceil \leq q \leq \lceil \frac{n+2k_h-2}{2} \rceil$, then  $R_{\mathcal{A}}(G)= -2t_h$;
\item[$\mathrm{(c)}$]
for $q= \lceil \frac{n+\sum\limits_{a=0}^{i}2k_{h-a}}{2} \rceil$, then $R_{\mathcal{A}}(G) \in(-2t_{h-i}, -2t_{h-i-1})$ $(0\leq i \leq h-1)$;
\item[$\mathrm{(d)}$]
for $\lceil \frac{n+\sum\limits_{a=0}^{i}2k_{h-a}+2}{2} \rceil \leq q \leq \lceil \frac{n+\sum\limits_{a=0}^{i+1}2k_{h-a}-2}{2} \rceil$, then $R_{\mathcal{A}}(G) = -2t_{h-i-1}$ $(0\leq i \leq h-2)$.
\end{itemize}
\end{itemize}
\end{thm}

Here is the remainder of the paper. In Section 2 we mainly give the proof of Theorem \ref{E-main-1-one}  which is decomposed into a series of lemmas. Especially, we determine the spectral distribution in the graphs with exactly one positive $\mathcal{A}$-eigenvalue. In Sections 3 and 4 we respectively provide the proofs for Theorems \ref{E-main-2-two} and \ref{E-HL-index} based on the results in previous section. In Section 5 we give some remarks and put forward several problems for further study.

\section{\large Graphs with exactly one positive $\mathcal{A}$-eigenvalue}

Throughout the paper, let $\mathscr{G}$ be the set of graphs with exactly one positive $\mathcal{A}$-eigenvalue. For two graphs $G$ and $H$, let $G \cup H$ be their {\it disjoint union}, and $H \subseteq G$ (or $H \nsubseteq G$) denote that $H$ is (or not) an induced subgraph of  $G$. We denote by $G \vee H$ the {\it join} obtained from $G \cup H$ by joining each vertex of $G$ to each one of $H$.

\begin{lem}[Cauchy Interlace Theorem]\label{cauchy}
Let $R$ be a real symmetric $n \times n$ matrix and let $S$ be a
principal submatrix of $R$ with order $m \times m$. Then, for
$i=1,2,\cdots,m$, $$\lambda_{n-m+i}(R)\leq \lambda_i(S)\leq
\lambda_i(R),$$ where $\lambda_1(R) \geq \lambda_2(R) \geq \cdots
\geq \lambda_n(R)$ and $\lambda_1(S) \geq \lambda_2(S) \geq \cdots
\geq \lambda_m(S)$ are respectively the eigenvalues of $R$ and $S$.
\end{lem}

\begin{lem}\label{E-sub}
Let $G$ be a connected graph and $H \subseteq G$. For any vertices $u,v\in V(H)$, if $\varepsilon_H(u) = \varepsilon_G(u)$ and $d_H(u,v)=d_G(u,v)$, then $\mathcal{A}(H)$ is a principal submatrix of $\mathcal{A}(G)$.
\end{lem}

\begin{proof}
For any two vertices $u,v\in V(H)$, we consider the $uv^{th}$ entry of $\mathcal{A}(H)$ and $\mathcal{A}(G)$, i.e., $\epsilon_{uv}(H)$ and $\epsilon_{uv}(G)$.
If $\epsilon_{uv}(G)\neq0$, we get $\epsilon_{uv}(G)=d_G(u,v)$ and then
$d_H(u,v)=d_G(u,v)= \min\{\varepsilon_G(u),\varepsilon_G(v)\}=\min\{\varepsilon_H(u),\varepsilon_H(v)\}$. So $\epsilon_{uv}(H)=d_H(u,v)=d_G(u,v)=\epsilon_{uv}(G)$.
If $\epsilon_{uv}(G)=0$, then $d_H(u,v)=d_G(u,v)< \min\{\varepsilon_G(u),\varepsilon_G(v)\}=\min\{\varepsilon_H(u),\varepsilon_H(v)\}$. So $\epsilon_{uv}(H)=0=\epsilon_{uv}(G)$.
From the above discussion, $\epsilon_{uv}(H)=\epsilon_{uv}(G)$ and so $\mathcal{A}(H)$ is a principal submatrix of $\mathcal{A}(G)$.
\end{proof}

The following corollary obviously follows from Lemmas \ref{cauchy} and \ref{E-sub}.

\begin{cor}\label{E-interlace}
Under the conditions in Lemma \ref{E-sub},  for $i=1,2,\cdots,n'$, $$\xi_{n-n'+i}(G)\leq \xi_i(H)\leq
\xi_i(G),$$ where $n' = |H|$ and $\xi_i(G)$ $(i=1,2,\ldots,n)$ is the $\mathcal{A}$-eigenvalue of $G$.
\end{cor}

\begin{table}[h!]
\begin{center}
\begin{tabular}{l l l}
\hline
Label & Graph & $\xi_2$\\
\hline
$F_1$ & $S(5,-3)$ & $4-\sqrt{15}$ \\
$F_2$ & $S(5,-2,2)$ & $0.138+$ \\
$F_3$ & $S(5,-1,2,2)$ & $0.152+$ \\
$F_{4}$ & $S(4,-4)$ & $\sqrt{1}{2}(9-\sqrt{73})$ \\
$F_{5}$ & $S(4,-3,2)$ & $0.238+$ \\
$F_{6}$ & $S(4,-2,2,2)$ & $0.248+$ \\
$F_{7}$ & $S(4,-1,2,2,2)$ & $0.259+$ \\
$F_{8}$ & $S(3,-5)$ & $5-2\sqrt{6}$ \\
$F_{9}$ & $S(3,-4,2)$ & $0.103+$ \\
$F_{10}$ & $S(3,-3,2,2)$ & $0.105+$ \\
$F_{11}$ & $S(3,-2,2,2,2)$ & $0.107+$ \\
$F_{12}$ & $S(3,-1,2,2,2,2)$ & $0.109+$ \\
\hline
\end{tabular}
\caption{The second largest $\mathcal{A}$-eigenvalues of $F_1$--$F_{12}$.}
\label{tab1}
\end{center}
\end{table}

 The subsequent lemma follows from Corollary \ref{E-interlace} and Table 1.

\begin{lem}\label{F1-12}
 Let $G = S(t_0, -p, t_1, \ldots, t_q) $ be a mixed extension with $p, q\geq 0$, $t_j \geq 2$ ($1 \leq j \leq q$). If $G \in \mathscr{G}$, then the graphs $F_1$--$F_{12}$ in  Table \ref{tab1} are not the induced subgraphs of $G$.
\end{lem}

As usual, let $C_n, P_n, K_n$ denote the {\it cycle}, {\it path} and {\it complete graph} of order $n$, respectively.

\begin{lem}\label{pos-diam-2}
For any $G \in \mathscr{G}$, $\diam(G) \leq 2$.
\end{lem}

\begin{proof}
Assume that $\diam(G) \geq 3$. Let $P_{d+1} = v_0v_1 \ldots v_{d-1}v_d$ be a path with length $\diam(G)=d$ of $G$. Then $\varepsilon_G(v_0) = \varepsilon_G(v_d) = d$ and $d-1 \leq \varepsilon_G(v_1), \varepsilon_G(v_{d-1}) \leq d$. Let us consider the following cases.

{\it Case 1.}  $\varepsilon_G(v_1) = d$ and $\varepsilon_G(v_{d-1}) \leq d$, or $\varepsilon_G(v_1) \leq d$ and $\varepsilon_G(v_{d-1}) = d$. Without loss of generality, for some $u \in V(G)$ set $\varepsilon_G(v_1) = d_G(v_1, u) = d = \varepsilon_G(u)$. Then the principal submatrix of $\mathcal{A}(G)$ indexed by $\{v_0, v_d, v_1, u\}$ is
$$
W_1 = \begin{pmatrix}        0        &      d             & 0 &  \epsilon_{v_0u}  \\
                             d        &      0             & 0 &  \epsilon_{v_du} \\
                             0        &      0             & 0 & d  \\
                   \epsilon_{uv_0} & \epsilon_{uv_d} & d & 0
                    \end{pmatrix}.
$$
By $d_G(v_i, u) \leq d$ we have  $\epsilon_{v_iu} \in \{0, d\}$ $(0 \leq i \leq d)$. Thereby, the principal submatrix of $\mathcal{A}(G)$ indexed by $\{v_0, v_d, v_1, u\}$ is one of the following matrices:
$$
W_2 = \begin{pmatrix}   0 & d & 0 & 0  \\
                        d & 0 & 0 & 0  \\
                        0 & 0 & 0 & d  \\
                        0 & 0 & d & 0
                    \end{pmatrix},
W_3 = \begin{pmatrix}   0 & d & 0 & 0  \\
                        d & 0 & 0 & d  \\
                        0 & 0 & 0 & d  \\
                        0 & d & d & 0
                    \end{pmatrix},
W_4 = \begin{pmatrix}   0 & d & 0 & d  \\
                        d & 0 & 0 & 0  \\
                        0 & 0 & 0 & d  \\
                        d & 0 & d & 0
                    \end{pmatrix},
W_5 = \begin{pmatrix}   0 & d & 0 & d  \\
                        d & 0 & 0 & d  \\
                        0 & 0 & 0 & d  \\
                        d & d & d & 0
                    \end{pmatrix}.
$$
A direct calculation shows that the second largest eigenvalues of the first three matrices above are respectively $d>0$, $\frac{-1+\sqrt{5}}{2}d>0$, $\frac{-1+\sqrt{5}}{2}d>0$, and that the characteristic polynomial of $W_5$ is $$\phi_{W_5}(\lambda) = (\lambda+d)f(\lambda), \;\; \mbox{where} \; f(\lambda) = \lambda^3 - d\lambda^2 - 3d^2\lambda + d^3.$$
For $d \geq 3$ we get $f(-d^2) = d^3 + 3 d^4 - d^5 - d^6 < 0$, $f(0) = d^3 > 0$, $f(d) = -2d^3 <0$ and $f(d^2) = d^3 - 3 d^4 - d^5 + d^6 > 0$.  Hence, the second eigenvalue $\xi_2(W_5)$ of $W_5$ is greater than 0. By Lemma \ref{cauchy}, we get $\xi_2(G) \geq \min\{\xi_2(W_{i}) \mid i = 2,3,4,5\} > 0$, a contradiction.

{\it Case 2.} $\varepsilon_G(v_1) = d-1 = \varepsilon_G(v_{d-1})$. Then, the principal submatrix of $\mathcal{A}(G)$ indexed by $\{v_0, v_d, v_1, v_{d-1}\}$ is
$$
W_6 = \begin{pmatrix}   0  &  d  &  0  & d-1  \\
                        d  &  0  & d-1 & 0    \\
                        0  & d-1 &  0  & 0    \\
                       d-1 &  0  &  0 & 0
                    \end{pmatrix}
$$
with the second largest eigenvalue $\xi_2(W_6) = \frac{1}{2}(-d+\sqrt{4 - 8d + 5d^2}) > 0$, and so $\xi_2(G) \geq \xi_2(W_6) > 0$ by Lemma \ref{cauchy}.

As proved above, $\diam(G) \leq 2$ for any $G \in \mathscr{G}$.

\end{proof}

\unitlength 3mm
\linethickness{0.4pt}
\ifx\plotpoint\undefined\newsavebox{\plotpoint}\fi
\begin{center}
\begin{picture}(41.114,15.074)(0,0)
\put(6.432,6){\makebox(0,0)[cc]{$P_4$}}
\put(.5,10.038){\circle*{1}}
\put(.553,8.741){\makebox(0,0)[cc]{$v_1$}}
\put(4.442,10.038){\circle*{1}}
\put(4.442,8.741){\makebox(0,0)[cc]{$v_2$}}
\put(8.436,10.091){\circle*{1}}
\put(8.436,8.741){\makebox(0,0)[cc]{$v_3$}}
\put(12.483,10.091){\circle*{1}}
\put(12.43,8.741){\makebox(0,0)[cc]{$v_4$}}
\put(.553,10.091){\line(1,0){11.93}}
\put(20.442,5){\makebox(0,0)[cc]{$C_4$}}
\put(17.605,7.074){\circle*{1}}
\put(17.605,5.9){\makebox(0,0)[cc]{$v_3$}}
\put(23.571,7.074){\circle*{1}}
\put(23.571,5.9){\makebox(0,0)[cc]{$v_4$}}
\put(17.649,12.996){\circle*{1}}
\put(17.561,14.5){\makebox(0,0)[cc]{$v_1$}}
\put(23.571,13.041){\circle*{1}}
\put(23.571,14.5){\makebox(0,0)[cc]{$v_2$}}
\put(17.561,13.085){\line(0,-1){5.966}}
\put(17.561,7.119){\line(1,0){6.099}}
\put(23.659,7.119){\line(0,1){6.01}}
\put(23.659,13.129){\line(-1,0){6.099}}
\put(34.707,6){\makebox(0,0)[cc]{$P_3 \cup K_1$}}
\put(28.637,10.094){\circle*{1}}
\put(28.593,11.5){\makebox(0,0)[cc]{$v_4$}}
\put(32.659,9.961){\circle*{1}}
\put(32.615,11.5){\makebox(0,0)[cc]{$v_3$}}
\put(36.681,10.006){\circle*{1}}
\put(36.592,11.5){\makebox(0,0)[cc]{$v_2$}}
\put(40.614,10.05){\circle*{1}}
\put(40.702,11.5){\makebox(0,0)[cc]{$v_1$}}
\put(32.659,10.05){\line(1,0){4.066}}
\put(28.682,10.05){\line(1,0){4.154}}
\put(13,2){{\bf Fig. 1}: Graphs in Lemma \ref{forbidden}.}
\end{picture}
\end{center}

\begin{lem}\label{forbidden}
No graph in $\mathscr{G}$ contains one in $\{P_4, C_4, P_3 \cup K_1\}$ as an induced subgraph.
\end{lem}

\begin{proof}
By Lemma \ref{pos-diam-2} it suffices to consider $\diam(G) = 2$. Assume by the contradiction that $P_4 = v_1v_2v_3v_4 \subseteq G$. Due to $\varepsilon_G(v_i) \leq 2 $ $(1 \leq i \leq 4)$,  $d_G(v_i, v_j) \leq 2$ for any $i ,j \in \{1, 2, 3, 4\}$. Hence, we get $d_G(v_1, v_3) = 2, d_G(v_1, v_4) = 2, d_G(v_2, v_4) = 2$ and $\varepsilon_G(v_i) = 2$ $(i = 1, 2, 3, 4)$. Thus, the principal submatrix of $\mathcal{A}(G)$ indexed by these four vertices are
$$
W_7 = \begin{pmatrix}   0 & 0 & 2 & 2  \\
                        0 & 0 & 0 & 2  \\
                        2 & 0 & 0 & 0  \\
                        2 & 2 & 0 & 0
       \end{pmatrix}
$$
with the second largest eigenvalue $\xi_2(W_7) = \sqrt{5}-1$. By Lemma \ref{cauchy} we get $\xi_2(G) \geq \xi_2(W_7)  > 0$, a contradiction. Hence, $P_4 \nsubseteq G$.

Set $C_4 \subseteq G$. Clearly, the principal submatrix of $\mathcal{A}(G)$ indexed by $V(C_4)$ is
$$
W_8 = \begin{pmatrix}   0 & 0 & 0 & 2  \\
                        0 & 0 & 2 & 0  \\
                        0 & 2 & 0 & 0  \\
                        2 & 0 & 0 & 0  \\
       \end{pmatrix}
$$
whose the second largest eigenvalue is $\xi_2(W_8) = 2$. Thus, $\xi_2(G) \geq 2 > 0$, a contradiction.

For $P_3 \cup K_1  \subseteq G$, the principal submatrix of $\mathcal{A}(G)$ indexed by those vertices is
$$
W_9 = \begin{pmatrix}   0 & 2 & 2 & 2  \\
                        2 & 0 & 0 & 2  \\
                        2 & 0 & 0 & 0  \\
                        2 & 2 & 0 & 0
       \end{pmatrix}
$$
with the second largest eigenvalue $\xi_2(W_9) \approx 0.6222$. So, $\xi_2(G) > 0$, a contradiction.

This completes the proof.
\end{proof}

Note that  the vertex set of $S(t_0, t_1, \ldots, t_k)$ is $V_0 \cup V_1 \cup \cdots \cup V_k$, where the corresponding set of $t_i$ is $V_i = \{v_{i1},v_{i2},\ldots,v_{it_i}\}$ ($0 \leq i \leq k$).
Let $G[v_1,\cdots,v_j]$ denote the subgraph of $G$ induced by $\{v_1,\cdots,v_j\}$.

\begin{lem}\label{add}
Let $G'=S(t_0, t_1, \ldots, t_k)$ with $k \geq 2$ and $t_i \geq 1$ ($1 \leq i \leq k)$. If $G' \in \mathscr{G}$ and $G \in \mathscr{G}$ is the connected graph obtained from $G'$ by adding a new vertex $w$, then $w$ must be adjacent to all vertices of $V_0$.
\end{lem}

\begin{proof}
For the mixed extension of star $S_{1,k}$, we get that any vertex in $V_0$ is adjacent to those ones in other $V_i's$ ($1\leq i \leq k$), and that the vertices in $V_i$ and the vertices in $V_j$ are mutually not adjacent ($1\leq i\neq j \leq k$).

Assume by way of contradiction that there exists some vertex $v_{01} \in V_0$ is not adjacent to $w$. For $v_{i1} \in V_i$ and $v_{j1} \in V_j$ ($i\neq j$), by the above statement we get  $v_{01}, v_{i1}$ and $v_{01},v_{j1}$ are adjacent. Due to  $P_1 \cup P_3 \nsubseteq G$ (see Lemma \ref{forbidden}), $w$ is adjacent to at least one vertex of $v_{i1}$ and $v_{j1}$. Then  $G[v_{01},v_{i1},w,v_{j1}] \cong C_4$ if $w$ is adjacent to both $v_{i1}$ and $v_{j1}$, and otherwise $G[v_{01},v_{i1},w,v_{j1}] \cong P_4$. This contradicts Lemma  \ref{forbidden}.

As shown above, $w$ must be adjacent to all vertices of $V_0$.
\end{proof}

\begin{prop}\label{E-one-positive}
Let $G \in \mathscr{G}$ with order $n$. Then $G$ is the mixed extension $S(t_0, t_1, \ldots, t_k)$, where $k \geq 1$ and $t_i \geq 1$ ($1 \leq i \leq k$).
\end{prop}

\begin{proof}

By Lemma \ref{pos-diam-2},  if $\diam(G) =1$ we get $G \cong K_n \cong S(t_0, n-t_0)$. We next set  $\diam(G) =2$.  If $G$ is a tree, by $\diam(G) = 2$ we get $G \cong S_{1,n-1} \cong S(1, 1,\ldots,1)$ ($n \geq 3$).

We now consider $G$ containing at least one cycle.  By $C_4, P_4 \nsubseteq G$ (Lemma \ref{forbidden}),  then $G$ only contains the triangle $C_3$. If $n=3$, then $G \cong C_3 \cong S(1,2)$. For $n=4$, clearly we get $G = S(1,1,2)$ or $S(2,1,1)$. When $n \geq 5$ we show the result by the induction. Assume that the theorem holds for the graphs with order $n-1$. As $G$ contains a cycle, there exists a vertex $w$ such that $G' = G-w$ is connected. If $\diam(G') =1$, then $G' \cong K_{n-1} \cong S(t_0, n-1-t_0) \in \mathscr{G}$. If $\diam(G') \geq 3$, then $P_4 \subseteq G' \subseteq G$  contradiction to Lemma \ref{forbidden}. Hence, $\diam(G') = 2$.

\begin{claim}
For $G \in \mathscr{G}$ and $w\in V(G)$, if $G' = G-w\subseteq G$ and $\diam(G') = 2$, then $G' \in \mathcal{G}$.
\end{claim}

{\bf Proof of the claim.}
From Lemma \ref{pos-diam-2} and $G \in \mathcal{G}$, we get $\diam(G) \leq 2$. Due to $\diam(G') = 2$, $\diam(G)=2$. Otherwise, $\diam(G)=1$, and so $\diam(G')=\diam(G-w)=1$, a contradiction. Thereby, for any two vertices $x,y \in V(G')$ we get $d_{G'}(x,y)=d_G(x,y)$.

For any $u\in V(G')$, we consider the values of $\varepsilon_{G'}(u)$ and $\varepsilon_{G}(u)$.
If $\varepsilon_{G'}(u)=1$, then we get $\varepsilon_{G}(u)=1$; otherwise, by $\diam(G)=2$ we get $\varepsilon_{G}(u)=2$, and hence $\varepsilon_{G'}(u)=2$, a contradiction.
If $\varepsilon_{G'}(u)=2$, we have $\varepsilon_{G}(u)=2$; or else, by $\diam(G)=2$ we have $\varepsilon_{G}(u)=1$, and thus $\varepsilon_{G'}(u)=1$, a contradiction again. Therefore, we obtain for any $u\in V(G')$ that $\varepsilon_{G'}(u)=\varepsilon_{G}(u)$.

From the above discussions, we know that $G$ and $G'$ satisfy the conditions of Corollary \ref{E-interlace}.
Hence, $$\xi_1(G) \geq \xi_1(G') \geq \xi_2(G) \geq \xi_2(G') \geq \cdots \geq \xi_{n-1}(G') \geq \xi_n(G)$$
which, along with $G \in \mathcal{G}$, leads to $G' \in \mathcal{G}$. \quad $\square$

By inductive hypothesis  and $G' \in \mathcal{G}$, we can set $G' \cong S(t_0', t_1', \ldots, t_k')$, where $\sum_{i=0}^{k} t'_i = n - 1$, $V'_i =\{v_{i1'}, \cdots, v_{it_i'}\}$, $k \geq 2$ and $t_i' \geq 1$ ($0 \leq i \leq k$). By Lemma \ref{add}, we get that  $w$ is adjacent to  all the vertices of $V'_0$.

If $w$ is not adjacent to any vertex of $V'_1 \cup \cdots \cup V'_k$, then $G \cong S(t_0', t_1', \ldots, t_k',1)$. If $w$ is adjacent to a vertex (say, $v_{s1'}$) of one set in $\{V'_1, \ldots, V'_k\}$ (say, $V'_s$), then $w$ must be adjacent to all vertices of $V'_s$. Otherwise, there exits a vertex $v_{st'_a} \in V'_s$ such that $w$ and $v_{st'_a}$ are not adjacent which implies $G[w, v_{s1'},v_{st'_a},v_{j1'}] \cong P_3 \cup K_1 \subseteq G$,  contradiction to Lemma \ref{forbidden}. Consequently, $G \cong S(t_0', t_1', \ldots, t_i'+1,\ldots, t_k')$. If $w$ is adjacent to the vertices of at least two sets in $\{V'_1, \ldots, V'_k\}$ (say, $V'_i$ and $V'_j$ ($0 \leq i \neq j \leq k$)), for $v_{i1'} \in V_i'$ and $v_{j1'} \in V_j'$ we conclude that $w$ must be adjacent to all vertices of $V'_1 \cup \cdots \cup V'_k$. Otherwise, there is a vertex $v \in V'_1 \cup \cdots \cup V'_k$ satisfying that $v$ and $w$ is not adjacent. In this case, $G[w,v,v_{i1'},v_{j1'}] \cong P_4$ if $v \in V'_i \cup V'_j$, or $G[w,v,v_{i1'},v_{j1'}] \cong P_3 \cup K_1$ if $v \notin V'_i \cup V'_j$, a contradiction. Therefore, $w$ is adjacent to all the vertices of $V'_1 \cup \cdots \cup V'_k$, and so $G\cong S(t_0'+1, t_1', \ldots, t_k')$.

This finishes the proof.
\end{proof}

\begin{prop}\label{positive}
Let $G= S(t_0, -p, t_1, \ldots, t_q)$ be a mixed extension with $p,q\geq 0$ and $t_j \geq 2$ $(1 \leq j \leq q)$. If $G \in \mathscr{G}$, then
\begin{itemize}
\item[$\mathrm{(i)}$]
$t_0 = 1$, $p + q \geq 1$;
\item[$\mathrm{(ii)}$]
$t_0 = 2$, $p, q \geq 0$;
\item[$\mathrm{(iii)}$]
$t_0 = 3$, $0 \leq q \leq 4 - p$ $(0 \leq p \leq 4)$;
\item[$\mathrm{(vi)}$]
$t_0 = 4$, $0 \leq q \leq 3 - p$ $(0 \leq p \leq 3)$;
\item[$\mathrm{(v)}$]
$t_0 \geq 5$, $0 \leq q \leq 2 - p$ $(0 \leq p \leq 2)$.
\end{itemize}
\end{prop}

\begin{proof}
Let $t_0 \geq 5$. If $p \geq 3$, then $F_1 \subseteq G$ contradiction to Lemma \ref{F1-12}. Hence, $p \leq 2$. If $p = 2$, then $q = 0$ by $F_2 \nsubseteq G$ (see Lemma \ref{F1-12}). If $p = 0, 1$, then $q \leq 2 - p$ by $F_3\nsubseteq G$. Set $t_0 = 4$.  Since $F_{4} \nsubseteq G$, then $p \leq 3$. Due to $F_{5}, F_{6}, F_{7} \nsubseteq G$, we get $q \leq 3 - p$. Let $t_0 = 3$. By $F_8 \nsubseteq G$ we obtain $p \leq 4$. In view of $F_i \nsubseteq G$ ($i=9,10,11,12$), we have $q \leq 4 - p$ ($p \leq 3$). Since the order of $G$ is at least $2$, we get $p, q \geq 0$ if $t_0 = 2$  and  $p + q \geq 1$ if $t_0 =1$.
\end{proof}

So far, we have shown the sufficiency of Theorem \ref{E-main-1-one}. We next show it is necessary. Write $\{t_1, t_2,  \cdots, t_q\}$ as the multiset  $\{k_1 \cdot t_1, \cdots, k_h \cdot t_h\}$, where $k_i$ is the number of $t_i$ ($1\leq i \leq h$). Clearly,  $S(t_0, -p, t_1, \ldots, t_q) = S(t_0,-p,k_1 \cdot t_1, \cdots, k_h \cdot t_h)$ and $\sum\limits_{i=1}^{h}k_i= q$.

\begin{prop}\label{ES-eigenvalue}
Let $G= S(t_0,-p,k_1 \cdot t_1, \cdots, k_h \cdot t_h)$ be a mixed extension with $p,q\geq 0$ and $t_j \geq 2$ ($1 \leq j \leq h$). Under the conditions $(\rm i)-(\rm v)$ in Proposition \ref{positive}, $G$ has exactly one positive eigenvalue. Furthermore,
\begin{itemize}
\item[$\mathrm{(i)}$]
$p+q \leq 1$. For $t_0 \geq 1$, then $G\cong K_n$ and
\begin{equation*}
{\rm Spec}_\mathcal{A}(G)=
\left
\{\begin{array}{ccccccccccc}
n-1& -1  \\

1         &n-1
\end{array}
\right\}.
\end{equation*}
\item[\rm{(ii)}]
$p+q \geq 2$ and $p \geq 1$. If $t_0=1, q =0$ or $t_0=3, q=4-p$ or $t_0=4, q=3-p$, we get
{\footnotesize
\begin{equation*}
{\rm Spec}_\mathcal{A}(G)=
\left
\{\begin{array}{ccccccccccc}
\xi_{1}& 0       & -1  & -2&\xi_5&-2t_h&\xi_{7}&\cdots&-2t_2&\xi_{2h+3}&-2t_1\\

1          &n-t_0-p-q+1&t_0-1&p-1&    1     &k_h-1&  1            &\cdots&k_2-1&1&k_1-1
\end{array}
\right\},
\end{equation*}}
where $\xi_5\in(-2t_h,-2)$, $\xi_{5+2i} \in (-2t_{h-i},-2t_{h-i+1})$ $(1\leq i \leq h-1)$.

Otherwise, for the other cases in conditions $(\rm i)-(\rm v)$ of Proposition \ref{positive},
{\footnotesize$$ {\rm Spec}_\mathcal{A}(G)=
\left
\{\begin{array}{cccccccccccc}
\xi_{1}& 0       & \xi_3& -1  & -2&\xi_6&-2t_h&\xi_8&\cdots&-2t_2&\xi_{2h+4}&-2t_1\\

1          &n-t_0-p-q&     1      &t_0-1&p-1&    1     &k_h-1& 1            &\cdots&k_2-1&1&k_1-1
\end{array}
\right\},
$$}
where $\xi_6\in(-2t_h,-2)$, $\xi_{6+2i} \in (-2t_{h-i},-2t_{h-i+1})$ $(1\leq i \leq h-1)$.

\item[$\mathrm{(iii)}$]
$p+q \geq 2$ and $p = 0$. If $t_0 =3, q=4$ or $t_0=4, q=3$, then
$$\footnotesize {\rm Spec}_\mathcal{A}(G)=
\left
\{\begin{array}{cccccccccccc}
\xi_{1}& 0  & -1       &-2t_h&\xi_{5}&-2t_{h-1}&\cdots&-2t_2&\xi_{2h+1}&-2t_1\\
1          &n-t_0-q+1 &t_0-1     &k_h-1&    1            &k_{h-1}-1&\cdots&k_2-1&1&k_1-1
\end{array}
\right\},$$
where $\xi_{3+2i} \in (-2t_{h-i},-2t_{h-i+1})$ $(1\leq i \leq h-1)$.

Otherwise, for the other cases in conditions $(\rm i)-(\rm v)$ of Proposition \ref{positive},
$$\footnotesize {\rm Spec}_\mathcal{A}(G)=
\left
\{\begin{array}{ccccccccccccc}
\xi_{1}& 0  & \xi_{3}&-1       &-2t_h&\xi_{6}&-2t_{h-1}&\cdots&-2t_2&\xi_{2h+2}&-2t_1\\
1          &n-t_0-q &1&t_0-1&k_h-1&    1            &k_{h-1}-1&\cdots&k_2-1&1&k_1-1
\end{array}
\right\},$$
where $\xi_{4+2i} \in (-2t_{h-i},-2t_{h-i+1})$ and $1\leq i \leq h-1$.
\end{itemize}
\end{prop}

\begin{proof}
$\mathrm{(i)}$
If $p+q \leq 1$, then $G=K_n$ is the mixed extension $S(t_0, n-t_0)$ with
$$ {\rm Spec}_{\mathcal{A}}(G)=
\left
\{\begin{array}{ccccccccccc}
n-1& -1  \\

1         &n-1
\end{array}
\right\}.
$$

$\mathrm{(ii)}$
Let $p+q \geq 2$ and $p\geq 1$. Labelling the vertices of $G$ properly, we get
\begin{equation}\label{G-matrix}
\mathcal{A}(G) = \begin{pmatrix}
                        J-I &    J    & J  & J  & \ldots & J \\
                        J   & 2(J-I)  & 2J & 2J & \ldots & 2J \\
                        J   &    2J   & 0  & 2J & \ldots & 2J \\
                       \vdots& \vdots &\vdots& \vdots & \ddots & \vdots \\
                        J   &    2J   & 2J & 2J & \ldots & 0
       \end{pmatrix},
\end{equation}
with  the $\mathcal{A}$-polynomial
\begin{eqnarray}\label{g1}
\phi(G,\lambda)=\lambda^{n-t_0-p-q}(\lambda+1)^{t_0-1}(\lambda+2)^{p-1}g_{t_0,-p,\cdots,t_q}(\lambda),
\end{eqnarray}
where
\begin{equation}\label{g}
g_{t_0,-p,\cdots,t_q}(\lambda)=\prod_{i=1}^{h}(\lambda+2t_{i})^{k_i-1}h(\lambda)
\end{equation}
and
\begin{multline*}
h(\lambda) = (\lambda^2-2p\lambda-t_0\lambda+3\lambda+t_0p-2t_0-2p+2)
\prod_{i=1}^{h}(\lambda+2t_i)-(\lambda+2) (2\lambda-t_0+2)\sum_{j=1}^{h}(k_jt_j\prod\limits_{\substack{i=1\\i \neq j}}^{h}(\lambda+2t_i)).
\end{multline*}
Obviously, the polynomial \eqref{g} has $q+2$ real roots denoted by $\lambda_1 \geq \cdots \geq \lambda_{q+2}$.
When $k_i \geq 2$, $-2t_{i}$ is a root of ($\ref{g}$) with multiplicity $k_i-1$ ($1\leq i \leq h$). Thus, the number of these roots is $q-h$, and hence the remaining $h+2$ roots of ($\ref{g}$) are the roots of $h(\lambda)$.
For $2\leq i\leq h$, we have
\begin{eqnarray*}
h(-2t_{i-1})h(-2t_{i})\!\!\!\!&=&\!\!\!\!k_{i-1}t_{i-1}k_it_i
(2t_{i-1}-2)(2t_i-2)(4t_{i-1}+t_0-2)(4t_{i}+t_0-2)
\\&\;&\times[\prod\limits_{j=1}^{i-2}(2t_j-2t_{i-1})(2t_j-2t_{i})]
(t_i-t_{i-1})(t_{i-1}-t_i)
\\&\;&\times [\prod\limits_{j=i+1}^{h}(2t_j-2t_{i-1})(2t_j-2t_{i})].
\end{eqnarray*}
Since $k_{i-1}t_{i-1}k_it_i (2t_{i-1}-2)(2t_i-2)(4t_{i-1}+t_0-2)(4t_{i}+t_0-2)>0$, $\prod\limits_{j=1}^{i-2}(2t_j-2t_{i-1})(2t_j-2t_{i}) > 0$,
$(t_i-t_{i-1})(t_{i-1}-t_i) < 0$ and $\prod\limits_{j=i+1}^{h}(2t_j-2t_{i-1})(2t_j-2t_{i}) > 0$, then $h(-2t_{i-1})h(-2t_{i}) <0$. Therefore, there is a root of $h(\lambda)$ that is in $(-2t_{i-1}, -2t_i)$ ($2\leq i\leq h$). At present, we have got $q-1$ roots of \eqref{g}. For  $G \in \mathscr{G}$ we get $\lambda_{1} >0$, and so we need  find the other two roots of $h(\lambda)$.

Since $p\geq 1$, then $h(-2t_h) = -k_ht_h(2t_h+2) (t_0+4t_h-2)\prod\limits_{i=1}^{h-1}(-2t_h+2t_i) <0$, $h(-2)= p(t_0+2)\prod\limits_{i=1}^{h}(2t_i-2) >0$. Hence, $\lambda_{3} \in (-2t_h, -2)$. A direct calculation shows  that
$h(-1)= t_0(p-1)\prod\limits_{i=1}^{h}(2t_i-1)+t_0\sum\limits_{j=1}^{h}(k_jt_j\prod\limits_{\substack{i=1\\ i \neq j}}^{h}(2t_i-1))$ and $h(-1) > 0$ for $p+q \geq 2$. As well,
{\small
\begin{equation}\nonumber
\begin{split}
h(0)&=(t_0p-2t_0-2p+2)
\prod\limits_{i=1}^{h}2t_i-2(-t_0+2)\sum\limits_{j=1}^{h}(2k_jt_j\prod\limits_{\substack{i=1\\i \neq j}}^{h}t_i)= 2^h(qt_0+pt_0-2t_0-2q-2p+2)\prod\limits_{i=1}^{h}t_i.
\end{split}
\end{equation}
}

We consider the conditions (i)--(v) in Proposition \ref{positive}.
If $t_0 = 1$, then $q \geq 0$. When $q=0$, $h(0)= 2^h(-q)\prod\limits_{i=1}^{h}t_i =0$. Otherwise, $h(0)<0$.
If $t_0 = 2$, $q\geq 0$ and so $h(0)= 2^h(-2)\prod\limits_{i=1}^{h}t_i <0$.
If $t_0 = 3$, then $0 \leq q \leq 4-p$.
$h(0)= 2^h(p+q-4)\prod\limits_{i=1}^{h}t_i <0$ for $q<4-p$ and
$h(0)=0$ for $q=4-p$.
If $t_0 = 4$, we get $0 \leq q \leq 3-p$. For $q<3-p$, $h(0)= 2^h(2p+2q-6)\prod\limits_{i=1}^{h}t_i <0$. Or else, $h(0)=0$.
If $t_0 \geq 5$, $0 \leq q \leq 2-p$ and hence $h(0)= 2^h(3p+3q-8)\prod\limits_{i=1}^{h}t_i <0$.

Consequently, for the cases $t_0=1, q =0$, or $t_0=3,q=4-p$, or $t_0=4,q=3-p$ we get $\lambda_{2} = 0$, and
{\footnotesize
$${\rm Spec}_\mathcal{A}(G)=
\left
\{\begin{array}{ccccccccccc}
\xi_{1}& 0       & -1  & -2&\xi_5&-2t_h&\xi_{7}&\cdots&-2t_2&\xi_{2h+3}&-2t_1\\

1          &n-t_0-p-q+1&t_0-1&p-1&    1     &k_h-1&  1            &\cdots&t_2-1&1&k_1-1
\end{array}
\right\},
$$}
where $\xi_5\in(-2t_h,-2)$ and $\xi_{5+2i} \in (-2t_{h-i},-2t_{h-i+1})$ $(1\leq i \leq h-1)$. Otherwise, $\lambda_{2} \in(-1,0)$  and thus
{\footnotesize
$${\rm Spec}_\mathcal{A}(G)=
\left
\{\begin{array}{cccccccccccc}
\xi_{1}& 0       & \xi_3& -1  & -2&\xi_6&-2t_h&\xi_8&\cdots&-2t_2&\xi_{2h+4}&-2t_1\\

1          &n-t_0-p-q&     1      &t_0-1&p-1&    1     &k_h-1& 1            &\cdots&t_2-1&1&k_1-1
\end{array}
\right\},
$$}
where $\xi_6\in(-2t_h,-2)$ and $\xi_{6+2i} \in (-2t_{h-i},-2t_{h-i+1})$ $(1\leq i \leq h-1)$.\medskip

$\mathrm{(iii)}$
Let $p+q \geq 2$ and $p=0$, i.e., $q\geq 2$. Labelling the vertices of $G$ properly, we get
\begin{equation}\label{G-matrix}
\mathcal{A}(G) = \begin{pmatrix}
                        J-I &  J  & J  & \ldots & J \\

                        J   &  0  & 2J & \ldots & 2J \\
                       \vdots&\vdots& \vdots & \ddots & \vdots \\
                        J   &   2J & 2J & \ldots & 0
       \end{pmatrix},
\end{equation}
with  the $\mathcal{A}$-polynomial $$\phi(G,\lambda)=\lambda^{n-t_0-q}(\lambda+1)^{t_0-1}\prod_{i=1}^{h}(\lambda+2t_{i})^{k_i-1}l(\lambda)$$ where
$$l(\lambda)= (\lambda-t_0+1)\prod_{i=1}^{h}(\lambda+2t_i)-(2\lambda-t_0+2)\sum_{j=1}^{h}(k_jt_j\prod\limits_{\substack{i=1\\i\neq j}}^{h}(\lambda+2t_i)).
$$
A straightforward calculation shows that
{\small
\begin{eqnarray*}
l(-1)\!\!\!\!&=&\!\!\!\! -t_0\prod\limits_{i=1}^{h}(2t_i-1)+t_0\sum\limits_{j=1}^{h}(k_jt_j\prod\limits_{\substack{i=1\\i \neq j}}^{h}(2t_i-1)) \\
\!\!\!\!&=&\!\!\!\! t_0[\sum\limits_{j=1}^{2}((k_j-1)t_j\prod\limits_{\substack{i=1\\i\neq j }}^{h}(2t_i-1))+ (t_1-t_2)\prod\limits_{i=3}^{h}(2t_i-1)+\sum\limits_{j=3}^{h}(k_jt_j\prod\limits_{\substack{i=1\\i\neq j }}^{h}(2t_i-1))+\prod\limits_{i=2}^{h}(2t_i-1)]\\
\!\!\!&>&\!\!\!0
\end{eqnarray*}
}
and
$l(0)= (-t_0+1)\prod\limits_{i=1}^{h}2t_i-(-t_0+2)\sum\limits_{j=1}^{h}(2k_jt_j\prod\limits_{\substack{i=1\\ i\neq j}}^{h}t_i)
=2^{h-1}(qt_0-2t_0-2q+2)\prod\limits_{i=1}^{h}t_i$.

Under the conditions (i)--(v) in Proposition \ref{positive} , we consider the following cases.
If $t_0 = 1$, then $q\geq 2$ and $l(0)= -2^{h-1}q\prod\limits_{i=1}^{h}t_i <0$.
If $t_0 = 2$, $l(0)= -2^{h}\prod\limits_{i=1}^{h}t_i <0$.
If $t_0 = 3$, we get $2 \leq q \leq 4$. For $2<q<4$, $l(0)= 2^{h-1}(q-4)\prod\limits_{i=1}^{h}t_i <0$. Otherwise, $l(0)=0$.
If $t_0 = 4$, $2 \leq q \leq 3$. $l(0)= 2^{h-1}(2q-6)\prod\limits_{i=1}^{h}t_i <0$ for $q<3$ and $l(0)=0$ for $q=3$.
If $t_0 \geq 5$, then $l(0)= -2^{h}\prod\limits_{i=1}^{h}t_i <0$.

Consequently, for $(t_0,q) = (3,4)$ or $(t_0,q) = (4,3)$
we have $\lambda_{2} = 0$, and
$$\footnotesize {\rm Spec}_{\mathcal{A}}(G)=
\left
\{\begin{array}{cccccccccccc}
\xi_{1}& 0  & -1       &-2t_h&\xi_{5}&-2t_{h-1}&\cdots&\xi_{2h+1}&-2t_1\\
1          &n-t_0-q+1 &t_0-1     &k_h-1&    1            &k_{h-1}-1&\cdots&1&k_1-1
\end{array}
\right\}$$
where $\xi_{3+2i} \in (-2t_{h-i},-2t_{h-i+1})$ $(1\leq i \leq h-1)$; Otherwise, $\lambda_{2} \in(-1,0)$ and
$$\footnotesize {\rm Spec}_{\mathcal{A}}(G)=
\left
\{\begin{array}{ccccccccccccc}
\xi_{1}& 0  & \xi_{3}&-1       &-2t_h&\xi_{6}&-2t_{h-1}&\cdots&\xi_{2h+2}&-2t_1\\
1          &n-t_0-q &1&t_0-1&k_h-1&    1            &-2t_{h-1}&\cdots&1&k_1-1
\end{array}
\right\}$$
with $\xi_{4+2i} \in (-2t_{h-i},-2t_{h-i+1})$ $(1\leq i \leq h-1)$.
\end{proof}

\noindent{\bf Proof of Theorem \ref{E-main-1-one}}. This theorem follows from Propositions \ref{E-one-positive}, \ref{positive} and \ref{ES-eigenvalue}.

\begin{re}
Recently, Sorgun and K\"{u}\c{c}\"{u}k \cite{sor-huc} independently studied the graphs with exactly one positive $\mathcal{A}$-eigenvalue. In their paper, $S(4,-t_1,t_2,t_3,t_4)$ and $S(3,-t_1,t_2,t_3,t_4,t_5)$ $(t_i\geq 1)$ are regarded as such graphs in \cite[Theorem 2.11 (iv)]{sor-huc}. However, in Table \ref{tab1}, we can see that $F_7=S(4,-1,2,2,2)$ and $F_{12}=S(3,-1,2,2,2,2)$ have more than one positive $\mathcal{A}$-eigenvalues.  On the other side, the first author of this paper, Jianfeng Wang,  has been discussing with Sezer Sorgun about this topic. Through communications, all the authors agreed to combine these two papers into the current version.
\end{re}

\section{\large Graphs with all but at most two $\mathcal{A}$-eigenvalues equal to $-2$ and $0$}

Let $K_{n_1, n_2, \ldots, n_l}$ be the complete multipartite graph with $n_1 \geq n_2 \geq \cdots \geq n_l$.

\begin{lem}{\rm \cite{{wang-li-CILS}}}\label{poly}
Let $G \cong K_{n_0}\vee K_{n_1,\ldots, n_l}$ with $n_0 \geq 1$, $n_r \geq 2$ and $l \geq 2$ $(1 \leq r \leq l)$. Then
$$\phi(G,\lambda)= (\lambda + 1)^{n_0 - 1}(\lambda + 2)^{n-n_0-l}[(\lambda - n_0 + 1)\prod\limits_{r=1}^{l}(\lambda-2n_r+2)-n_0 \sum\limits_{r=1}^{l}\prod\limits_{\substack{s=1\\r\neq s}}^{l}n_r(\lambda - 2n_s + 2)].$$
\end{lem}

\begin{lem}{\rm \cite{{wang-li-CILS}}}\label{least}
Let $G$ be a graph with order $n$ and the least $\mathcal{A}$-eigenvalue $\xi_n(G)$. Then  $\xi_n(G) = -2$  if and only if
\begin{itemize}
\item[$\mathrm{(i)}$]
$G \cong K_{n_1, n_2, \ldots, n_l}$, where $l\geq 2$ and $n_r \geq 2$ $(1 \leq r \leq l)$;
\item[$\mathrm{(ii)}$]
$G \cong K_{n_0}\vee K_{n_1,\ldots, n_l}$, where $n_r \geq 2$ $(1 \leq r \leq l)$ and  $2 \leq l \leq 4$ if $n_0 = 1$, $2 \leq l \leq 3$ if $n_0 = 2$ or  $l = 2$ if $n_0 \geq 3$.
\end{itemize}
\end{lem}

Note that the connected graphs with order $n \geq 2$ have at least one positive $\mathcal{A}$-eigenvalue. Let $\mathscr{S}$ denote the set of connected graphs with all but at most two $\mathcal{A}$-eigenvalues equal to $-2$ and $0$. For $G \in \mathscr{S}$,  $G$ might have exactly two positive $\mathcal{A}$-eigenvalues, or only one positive $\mathcal{A}$-eigenvalue, or one positive and one negative $\mathcal{A}$-eigenvalue different from $-2$ and $0$. In the next lemma, we consider the first case.

\begin{lem}\label{two}
Let $G \in \mathscr{S}$ with order $n$. Then $G$ has two positive $\mathcal{A}$-eigenvalues if and only if $G \cong K_{n_1,n_2}$ with $n_1,n_2 \geq 2$.
\end{lem}

\begin{proof}
If  $G \in \mathscr{S}$ has exactly two positive $\mathcal{A}$-eigenvalue different from $0$ and $-2$, then the least $\mathcal{A}$-eigenvalue of $G$ is $-2$.  By Lemma \ref{least}, we discuss the following two cases.

{\it Case 1.} $G \cong K_{n_1, n_2, \ldots, n_l}$, where $l\geq 2$ and $n_r \geq 2$ $(1 \leq r \leq l)$.  Since
${\rm Spec}_{\mathcal{A}}(K_{n_1, n_2, \ldots, n_l})=\{2(n_1 - 1),2(n_2 - 1),\ldots,2(n_l - 1),-2^{(n-l)}\}$, then $G$ has at least two positive $\mathcal{A}$-eigenvalues. Thereby,  $G \cong K_{ n_1,n_2}$ with $n_1,n_2 \geq 2$.\smallskip

{\it Case 2.} $G \cong K_{n_0}\vee K_{n_1,\ldots, n_l}$, where $n_0 \geq 1$, $l\geq 2$ and $n_1\geq \cdots \geq n_l \geq 2$ $(1 \leq r \leq l)$.
By Lemma \ref{poly}, we get
\begin{equation}\label{fnnn}
\phi(G,\lambda)= (\lambda + 1)^{n_0 - 1}(\lambda + 2)^{n-n_0-l}f_{n_0,n_1,\ldots, n_l}(\lambda),
\end{equation}
where $f_{n_0,n_1,\ldots, n_l}(\lambda)=(\lambda - n_0 + 1)\prod\limits_{r=1}^{l}(\lambda-2n_r+2)-n_0 \sum\limits_{r=1}^{l}\prod\limits_{\substack{s=1\\r\neq s}}^{l}n_r(\lambda - 2n_s + 2)]$.
If $n_0 \geq 3$, then $G$ has at least three $\mathcal{A}$-eigenvalues (i.e., one positive  $\mathcal{A}$-eigenvalue and at least two $\mathcal{A}$-eigenvalues $-1$) different from $-2$ and $0$, a contradiction. Hence, $n_0 \leq 2$. Due to Lemma \ref{least}, we have $2 \leq l \leq 4$ and thus distinguish the following cases.

{\it Subcase 2.1.}  $l = 2$.  From Lemma \ref{least}(ii) it follows that $n_0 \geq 1$. By \eqref{fnnn} we get $$f_{n_0,n_1,n_2}(\lambda)=(\lambda-n_0+1)(\lambda-2n_1+2)(\lambda-2n_2+2)-n_0n_1(\lambda-2n_2+2)-
n_0n_2(\lambda-2n_1+2).$$
By calculations we get
$f_{n_0,n_1,n_2}(-2) = -4n_1n_2 < 0$,
$f_{n_0,n_1,n_2}(-1) = (n_1+n_2-1)n_0 > 0$,
$f_{n_0,n_1,n_2}(2n_2-2)=2n_0n_2(n_1-n_2) \geq 0$,
$f_{n_0,n_1,n_2}(2n_1-2)=2n_0 n_1(n_2-n_1) \leq 0$ and
$f_{n_0,n_1,n_2}(2(n_0+n_1+n_2)-2)=2n_1^2(4n_0+3n_2)+2n_0[3n_2^2+2(n_0-1)n_0+n_2(5n_0-2)]+ 2n_1[4n_2^2+n_0(5n_0-2)+2n_2(5n_0-1)]>0$. Hence, $G$ has more than two $\mathcal{A}$-eigenvalues different from $-2$ and $0$ (i.e., $\xi_1\in[2n_1-2,2(n_0+n_1+n_2)-2), \xi_2\in [2n_2-2,2n_1-2), \xi_{n_0+2}\in(-2,-1)$), a contradiction.\smallskip

{\it Subcase 2.2.} $l = 3$.  By Lemma \ref{least}(ii) we get $n_0 = 1$ or 2. If  $n_0 = 1$, then $G=K_{1}\vee K_{n_1,n_2,n_3}$ and
\begin{equation}\nonumber
\begin{split}
f_{1,n_1,n_2,n_3}(\lambda)=&\lambda(\lambda-2n_1+2)(\lambda-2n_2+2)(\lambda-2n_3+2)
-n_0n_1(\lambda-2n_1+2)(\lambda-2n_2+2)\\
&-n_0n_2(\lambda-2n_1+2)(\lambda-2n_3+2)
-n_0n_3(\lambda-2n_2+2)(\lambda-2n_3+2).
\end{split}
\end{equation}
We get
$f_{1,n_1,n_2,n_3}(-2) = 4n_1n_2n_3> 0$,
$f_{1,n_1,n_2,n_3}(2n_3-2)= 4n_3(n_1-n_3)(n_3-n_2)\leq0$,
$f_{1,n_1,n_2,n_3}(2n_2-2)= 4n_2(n_1-n_2)(n_2-n_3)\geq0$,
$f_{1,n_1,n_2,n_3}(2n_1-2)=-4n_1(n_1-n_2)(n_1-n_3)\leq0$,
$f_{1,n_1,n_2,n_3}(2(n_1+n_2+n_3)-2)= 4(n_1^3(-1+4n_2+4n_3)+
n_1^2(n_2+n_3)(8n_2+8n_3-5)+(n_2+n_3)(4n_2(n_3-1)n_3-n_3^2+n_2^2(4n_3-1))+
n_1(4n_2^3+n_3^2 (4n_3-5)+n_2n_3(16n_3-11)+n_2^2(16n_3-5)))>0$. Therefore, $G$ has more than two $\mathcal{A}$-eigenvalues different from $-2$ and $0$ (i.e., $\xi_1\in[2n_1-2,2(n_1+n_2+n_3)-2), \xi_2\in [2n_2-2,2n_1-2), \xi_3\in [2n_3-2,2n_2-2), \xi_{n_0+3}\in(-2,2n_3-2)$), a contradiction.\smallskip

If $n_0 = 2$, then $G = K_{2}\vee K_{n_1,n_2,n_3}$.  It is obvious that $G$ and $K_{1}\vee K_{n_1,n_2,n_3}$ satisfy the conditions of Lemma \ref{E-sub}.
So, $\mathcal{A}(K_{1}\vee K_{n_1,n_2,n_3})$ is the principle submatrix of $\mathcal{A}(K_{2}\vee K_{n_1,n_2,n_3})$. From Corollary \ref{E-interlace}, by the above discussion  we get $\xi_{1} \geq 2n_1-2, \xi_{2} \geq 2n_2-2$ and $\xi_3 \geq 2n_3-2$, a contradiction.\smallskip

{\it Subcase 2.3.} $l = 4$. In view of Lemma \ref{least}(ii) we get $n_0 = 1$ and $G = K_{1}\vee K_{n_1,n_2,n_3,n_4}$.  Clearly, the graphs $G$ and $K_{1}\vee K_{n_1,n_2,n_3}$ satisfy the conditions of Lemma \ref{E-sub}. Thus,  $\mathcal{A}(K_{1}\vee K_{n_1,n_2,n_3})$ is the principle submatrix of $\mathcal{A}(K_{1}\vee K_{n_1,n_2,n_3,n_4})$. From Corollary \ref{E-interlace}, by Case 2.2 we obtain $\xi_{1} \geq 2n_1-2, \xi_{2} \geq 2n_2-2$ and $\xi_3 \geq 2n_3-2$, a contradiction.\smallskip

As proved above, we get that $G \cong K_{n_1,n_2}$ ($n_1,n_2 \geq 2$) if $G \in \mathscr{S}$ has two positive $\mathcal{A}$-eigenvalues.
\end{proof}

We next identify the graphs $G \in \mathscr{S}$ which have only one positive $\mathcal{A}$-eigenvalue, or one positive and one negative $\mathcal{A}$-eigenvalue. At this moment, $G$ has exactly one positive $\mathcal{A}$-eigenvalue, and so $G \in \mathscr{G}$ defined in Section 2.

\begin{lem}\label{one}
$G \in \mathscr{S} \cap \mathscr{G}$ if and only if $G$ is the star $S_{1,p}=S(1,-p)$ with $p \geq 1$.
\end{lem}

\begin{proof}
For $G \in \mathscr{G}$, by Theorem \ref{E-main-1-one} we get $G \cong S(t_0, -p, t_1, \ldots, t_q)$ with  those restricted conditions.  Recall, $t_1\geq t_2\geq \cdots \geq t_q \geq 2$. If $t_0 \geq 3$, then by \eqref{g} we deduce that $G$ has at least two $\mathcal{A}$-eigenvalues $-1$, a contradiction.   Hence, $t_0 \leq 2$.\smallskip

{\it Case 1.} $t_0 = 1$. By Theorem \ref{E-main-1-one}(i) we get $p+q \geq 1$ with $p,q \geq 0$. \smallskip

{\it Subcase 1.1.} $p = 0$. Then $q\geq 1$, and thus $G \cong S(1,t_1,t_2,\ldots,t_q)$ with $t_j \geq 2$ ($1 \leq j \leq q$). If $q = 1$, then $G \cong K_n$. $G \in \mathscr{S}$ for $n=2$(i.e., $K_2= S(1,-1)$),  but $G\notin \mathscr{S}$ for others. If $q \geq 2$, set $g_{t_0,-p,\cdots,t_q}(\lambda)$ be defined in \eqref{g}.\smallskip

If $q = 2$, then $G \cong S(1,t_1,t_2)$. From Proposition \ref{ES-eigenvalue}(iii) it follows that $G$ has three $\mathcal{A}$-eigenvalues different from $-2$ and $0$ (i.e., $\xi_1 >0, \xi_{n-1}\in (-1,0), \xi_{n} \in [-2t_1,-2t_2)$), a contradiction.\smallskip

If $q = 3$, then $G \cong S(1,t_1,t_2,t_3)$. By Proposition \ref{ES-eigenvalue}(iii) we get that $G$ has four $\mathcal{A}$-eigenvalues  different from $-2$ and $0$ (i.e., $\xi_1 >0, \xi_{n-2}\in (-1,0), \xi_{n-1}\in [-2t_2,-2t_3), \xi_{n} \in [-2t_1,-2t_2)$), a contradiction.\smallskip

If $q \geq 4$, then $G \cong S(1,t_1,\cdots,t_q)$. Clearly, $G$ and $S(1,t_1,t_2,t_3)$ share the conditions of Lemma \ref{E-sub}.
So, $\mathcal{A}(S(1,t_1,t_2,t_3))$ is the principal submatrix of $\mathcal{A}(S(1,t_1,\cdots,t_q))$. By Corollary \ref{E-interlace} and the above conclusion, we get $\xi_1>0,\xi_{n-1} \leq -2t_3$ and $\xi_{n} \leq -2t_2$, a contradiction.\smallskip

{\it Subcase 1.2.} $p = 1$. Then $q \geq 0$. If $q =0$, then $G = S(1,-1) \in \mathscr{S} \cap \mathscr{G}$.\smallskip

If $q =1$, then $G \cong S(1,-1,t_1)$ and $G$ has three $\mathcal{A}$-different from $-2$ and $0$ (i.e., $\xi_1>0, \xi_{n-1}\in(-1,0), \xi_n\in(-2t_1,-2)$) by Proposition \ref{ES-eigenvalue}(ii), a contradiction.\smallskip

If $q = 2$, then $G \cong S(1,-1,t_1,t_2)$. From Proposition \ref{ES-eigenvalue}(ii) it follows that $G$ has four $\mathcal{A}$-eigenvalues different from $-2$ and $0$ (i.e., $\xi_1>0, \xi_{n-2}\in(-1,0), \xi_{n-1}\in[-2t_2,-2), \xi_n\in[-2t_1,-2t_2)$),  a contradiction.\smallskip

If $q \geq 3$, then $G \cong S(1,-1,t_1,\cdots,t_q)$. Obviously, $G$ and $S(1,-1,t_1,t_2)$ possess the conditions of Lemma \ref{E-sub}.
Hence, $\mathcal{A}(S(1,-1,t_1,t_2))$ is the principal submatrix of $\mathcal{A}(S(1,-1,t_1,\cdots,t_q))$. By Corollary \ref{cauchy} and the above case, we get $\xi_1>0, \xi_{n-1} <-2$ and $\xi_{n} \leq -2t_2$, a contradiction.\smallskip

{\it Subcase 1.3.} $p \geq 2$. Then $q \geq 0$. If $q = 0$, then $G\cong S(1,-p)$ and $${\rm Spec}_{_{\mathcal{A}}}(S(1,-p)) = \{\sqrt{p^2-p+1}+p-1, 0^{(n-t_0-p-q)},-\sqrt{p^2-p+1}+p-1,-2^{(p-1)}\}.$$ Thus, the star $S(1,-p) \in \mathscr{S} \cap \mathscr{G}$.\smallskip

If $q = 1$, then $G \cong S(1,-p,t_1)$. By Proposition \ref{ES-eigenvalue} we get $G$ has three $\mathcal{A}$-eigenvalues different from $-2$ and $0$ (i.e. $\xi_1>0, \xi_{n-p}\in(-1,0), \xi_n\in(-2t_1,-2)$),  a contradiction.\smallskip

If $q \geq 2$, then $G \cong S(1,-p,t_1,\cdots,t_q)$. Similarly, $\mathcal{A}(S(1,-1,t_1,t_2))$ is the principal submatrix of $\mathcal{A}(S(1,-p,t_1,\cdots,t_q))$. Analogously to the last case of Subcase 1.2 we get a conflict.\smallskip

{\it Case 2.}  $t_0 = 2$. By Theorem \ref{E-main-1-one} we get $p,q \geq 0$.

{\it Subcase 2.1.} $p = 0$. Then $q \geq 1$ and $G \cong S(2,t_1,t_2,\ldots,t_q)$. If $q = 1$, then $G \cong K_n \notin \mathscr{S}$. If $q = 2$, then $G \cong S(2,t_1,t_2)$ and $G$ has three $\mathcal{A}$-eigenvalues different from $-2$ and $0$ (i.e., $\xi_1 >0, \xi_{n-1}\in (-1,0), \xi_{n} \in [-2t_1,-2)$) by Proposition \ref{ES-eigenvalue}(iii), a contradiction. \smallskip

If $q \geq 3$, then $G \cong S(2,t_1,\cdots,t_q)$.  Similarly, $\mathcal{A}(S(1,t_1,t_2,t_3))$ is the principal submatrix of $\mathcal{A}(S(2,t_1,\cdots,t_q))$. Analogously to Subcase 1.1 (when $q=3$), we get a contradiction. \smallskip

{\it Subcase 2.2.} $p = 1$. So $q\geq 0$. If $q = 0$, then $G \cong S(2,-1) \cong C_3 \not\in \mathscr{S}$. If $q = 1$, then $G \cong S(2,-1,t_1)$. By Proposition \ref{ES-eigenvalue}(ii) we get that $G$ has three $\mathcal{A}$-eigenvalues different from $-2$ and $0$ (i.e., $\xi_1>0, \xi_{n-1} \in (-1,0), \xi_n\in(-2t_1,-2)$), a contradiction.  If $q \geq 2$, then $G \cong S(2,-1,t_1,\cdots,t_q)$. Similarly,  $\mathcal{A}(S(1, -1,t_1,t_2))$ is the principal submatrix of $\mathcal{A}(S(2,-1,t_1,\cdots,t_q))$. Analogously to Subcase 1.2 (when $q=2$), we get $G \notin \mathscr{S}$.\smallskip

{\it Subcase 2.3.} $p \geq 2$. If $q = 0$, then $G \cong S(2, -p)$ and we can get $\mathcal{A}$-eigenvalues of $G$ from Proposition \ref{ES-eigenvalue}(ii). Then $G$ has three $\mathcal{A}$-eigenvalues different from $-2$ and $0$ (i.e., $\xi_1>0,\xi_{n-p-1} \in (-1,0),\xi_{n-p}= -1$), a contradiction.\smallskip

If $q = 1$, then $G \cong S(2,-p,t_1)$. Hence, $G$ has three $\mathcal{A}$-eigenvalues different from $-2$ and $0$ (i.e., $\xi_1>0, \xi_{n-p-1} \in (-1,0), \xi_n \in(-2t_1,-2)$) by Proposition \ref{ES-eigenvalue}(ii), a contradiction.\smallskip

If $q \geq 2$, then $G \cong S(2,-p,t_1,\cdots,t_q)$. Similarly, $\mathcal{A}(S(1,-1,t_1,t_2))$ is the principal submatrix of $\mathcal{A}(S(2,-p,t_1,\cdots,t_q))$. In a similar way, $G \not\in \mathscr{S}$.\smallskip

As discussed above,  $G \in \mathscr{S} \cap \mathscr{G}$ if and only if $G \cong S(1,-p)$ with $p \geq 2$.
\end{proof}

The following corollary follows from the proof of Lemma \ref{one}.

\begin{cor}
No graph in $\mathscr{S}$ has only one positive $\mathcal{A}$-eigenvalue different from $-2$ and $0$.
\end{cor}

Clearly, the star $S_{1,p} = S(1,-p)$ is a special kind of complete bipartite graphs $K_{n_1,n_2} \cong S(-n_1,-n_2)$.\medskip

\noindent{\bf Proof of Theorem \ref{E-main-2-two}.}
This theorem follows from  Lemma \ref{two} and Corollary \ref{one}.

\section{\large HL-index of graphs with exactly one $\mathcal{A}$-eigenvalue}

In this section, we give a proof for Theorem \ref{E-HL-index}.\medskip

\noindent{\bf Proof of Theorem \ref{E-HL-index}}.
Clearly, the graph $G$ satisfies the conditions (i)-(v) in Proposition \ref{positive}. Therefore, $\xi_{\lceil\frac{n+1}{2}\rceil}  \leq \xi_{\lfloor\frac{n+1}{2}\rfloor} \leq 0$ and $R_{\mathcal{A}}(G) = \max\{|\xi_{\lfloor\frac{n+1}{2}\rfloor}|, |\xi_{\lceil\frac{n+1}{2}\rceil}|\} = |\xi_{\lceil\frac{n+1}{2}\rceil}|$.\smallskip

From Proposition \ref{ES-eigenvalue}(i), for $p+q \leq 1$ we get $R_{\mathcal{A}}(G) =1$. For $p+q \geq 2$, by Proposition \ref{ES-eigenvalue}(ii) and (iii) we consider the following cases.\smallskip

{\it Case 1.} $n=t_0+p+q$. Due to $n\geq 2$, then  $\lceil\frac{n+1}{2}\rceil \geq 2 =1+(n-t_0-p-q)+1$, and thus $p+q \geq \lceil \frac{n-2t_0+2}{2} \rceil$. If $p+q = \lceil \frac{n-2t_0+2}{2} \rceil$, then $\lceil\frac{n+1}{2}\rceil = 1+(n-t_0-p-q)+1$. For $t_0=1, q=0$ or $t_0=3,q=4-p$ or $t_0=4,q =3-p$, $R_{\mathcal{A}}(G)= |\xi_2|= 0$; otherwise, for the other cases in conditions (i)-(v) of Proposition \ref{positive}, we get $R_{\mathcal{A}}(G) = |\xi_2| \in (0,1)$.\smallskip

If $\lceil \frac{n-2t_0+4}{2} \rceil \leq p+q \leq \lceil \frac{n}{2} \rceil$, then $t_0 \geq 2$ and $1+(n-t_0-p-q)+1+1 \leq \lceil\frac{n+1}{2}\rceil \leq 1+(n-t_0-p-q)+1+(t_0-1)$. So, $R_{\mathcal{A}}(G)= |\xi_3|= 1$ by Proposition \ref{ES-eigenvalue}(ii) and (iii).\smallskip

If $p+q \geq \lceil \frac{n+2}{2} \rceil$, we get that $\lceil\frac{n+1}{2}\rceil \geq 1+(n-t_0-p-q)+1+(t_0-1)+1$.\smallskip

{\it Case 1.1.} If $t_0 =1$, $q=0$. Then $p\geq 2$ and $q =0 \leq \lceil \frac{n-2}{2} \rceil$. Hence, $1+(n-t_0-p-q)+1+(t_0-1)+1 \leq \lceil\frac{n+1}{2}\rceil \leq 1+(n-t_0-p-q)+1+(t_0-1)+(p-1)$ and $R_{\mathcal{A}}(G) = |\xi_3| = 2$ by Proposition \ref{ES-eigenvalue}(ii). If $t_0=3,q=4-p$, then $p+q =4 \geq \lceil \frac{n+2}{2} \rceil$ which leads to $n\leq 6 < t_0+p+q=7$, a contradiction.
For $t_0=4,q =3-p$, then $p+q =3 \geq \lceil \frac{n+2}{2} \rceil$ which results in $n\leq 4 < t_0+p+q=7$, a contradiction.

{\it Case 1.2.}
For the other cases in conditions (i)-(v) of Proposition \ref{positive}, if we assume  $t_0 \geq 5$, then $p+q =2\geq \lceil \frac{n+2}{2} \rceil$ implying $n\leq 2 < t_0+p+q$; if $t_0 =4$,
$p+q <3 \geq \lceil \frac{n+2}{2} \rceil$ showing $n< 4 < t_0+p+q$; if $t_0 =3$,
$p+q <4 \geq \lceil \frac{n+2}{2} \rceil$ indicating $n< 6 < t_0+p+q$. For any previous case, we always obtain a contradiction. Hence, $t_0 \leq 2$.

{\it Case 1.2.1.}
$t_0 =1, p\geq 2$. From Proposition \ref{ES-eigenvalue}(ii), we distinguish the following cases.
If $q \leq \lceil \frac{n-2}{2} \rceil$, then $R_{\mathcal{A}}(G) = |\xi_3| = 2$.
If $q = \left\lceil \frac{n}{2} \right\rceil$, we get $\lceil\frac{n+1}{2}\rceil = 1+(n-t_0-p-q)+1+(t_0-1)+(p-1)+1$, and hence $R_{\mathcal{A}}(G) = |\xi_4| \in (2, 2t_h)$.
If $\lceil \frac{n+2}{2} \rceil \leq q \leq \lceil \frac{n+2k_h-2}{2}\rceil$, then $R_{\mathcal{A}}(G) = |\xi_5| = 2t_h$. If $q= \lceil \frac{n+\sum\limits_{a=0}^{i}2k_{h-a}}{2} \rceil$, then
$1+(n-t_0-p-q)+1+(t_0-1)+(p-1)+1+1 \leq \lceil\frac{n+1}{2}\rceil \leq 1+(n-t_0-p-q)+1+(t_0-1)+(p-1)+1+(k_h-1)$ and therefore $R_{\mathcal{A}}(G) = |\xi_{6+2i}| \in(2t_{h-i}, 2t_{h-i-1})$ $(0\leq i \leq h-1)$.
If $\lceil \frac{n+\sum\limits_{a=0}^{i}2k_{h-a}+2}{2} \rceil \leq q \leq \lceil \frac{n+\sum\limits_{a=0}^{i+1}2k_{h-a}-2}{2} \rceil$, then $\lceil\frac{n+1}{2}\rceil = 1+(n-t_0-p-q)+1+(t_0-1)+(p-1)+1+\sum\limits_{a=0}^i(k_{h-a}-1)+i+1$, and thus
$R_{\mathcal{A}}(G)= |\xi_{7+2i}| = 2t_{h-i-1}$ $(0\leq i \leq h-2)$.\smallskip

{\it Case 1.2.2.}
$t_0 =1, p = 1$. Then $q \geq \lceil \frac{n}{2} \rceil \geq 1$ and we discuss the following cases from Proposition \ref{ES-eigenvalue}(ii).
If $q = \lceil \frac{n}{2} \rceil$, then $R_{\mathcal{A}}(G) = |\xi_3| \in (2, 2t_h)$.
If $\lceil \frac{n+2}{2} \rceil \leq q \leq \lceil \frac{n+2k_h-2}{2}\rceil$,
$R_{\mathcal{A}}(G) = |\xi_4| = 2t_h$.
If $q= \lceil \frac{n+\sum\limits_{a=0}^{i}2k_{h-a}}{2} \rceil$,
$R_{\mathcal{A}}(G) = |\xi_{5+2i}| \in(2t_{h-i}, 2t_{h-i-1})$ $(0\leq i \leq h-1)$.
If $\lceil \frac{n+\sum\limits_{a=0}^{i}2k_{h-a}+2}{2} \rceil \leq q \leq \lceil \frac{n+\sum\limits_{a=0}^{i+1}2k_{h-a}-2}{2} \rceil$, then
$R_{\mathcal{A}}(G)= |\xi_{6+2i}| = 2t_{h-i-1}$ $(0\leq i \leq h-2)$.\smallskip

{\it Case 1.2.3.}
$t_0 =1, p = 0$. Then $q \geq \lceil \frac{n+2}{2} \rceil \geq 2$. By Proposition \ref{ES-eigenvalue}(iii) we have the following cases.
If $\lceil \frac{n+2}{2} \rceil \leq q \leq \lceil \frac{n+2k_h-2}{2}\rceil$, then $R_{\mathcal{A}}(G) = |\xi_3| = 2t_h$.
If $q= \lceil \frac{n+\sum\limits_{a=0}^{i}2k_{h-a}}{2} \rceil$, then
$R_{\mathcal{A}}(G) = |\xi_{4+2i}| \in(2t_{h-i}, 2t_{h-i-1})$ $(0\leq i \leq h-1)$.
If $\lceil \frac{n+\sum\limits_{a=0}^{i}2k_{h-a}+2}{2} \rceil \leq q \leq \lceil \frac{n+\sum\limits_{a=0}^{i+1}2k_{h-a}-2}{2} \rceil$, then
$R_{\mathcal{A}}(G)= |\xi_{5+2i}| = 2t_{h-i-1}$ $(0\leq i \leq h-2)$.\smallskip

{\it Case 1.2.4.}
$t_0=2, p\geq 2$. Similarly to {\it Case 1.2.1}, we get the following cases.
If $q \leq \lceil \frac{n-2}{2} \rceil$, then $R_{\mathcal{A}}(G) = |\xi_4| = 2$.
If $q = \lceil \frac{n}{2} \rceil$, then $R_{\mathcal{A}}(G) = |\xi_5| \in (2, 2t_h)$.
If $\lceil \frac{n+2}{2} \rceil \leq q \leq \lceil \frac{n+2k_h-2}{2}\rceil$, then $R_{\mathcal{A}}(G) = |\xi_6| = 2t_h$.
If $q= \lceil \frac{n+\sum\limits_{a=0}^{i}2k_{h-a}}{2} \rceil$, then $R_{\mathcal{A}}(G) = |\xi_{7+2i}| \in(2t_{h-i}, 2t_{h-i-1})$ $(0\leq i \leq h-1)$.
If $\lceil \frac{n+\sum\limits_{a=0}^{i}2k_{h-a}+2}{2} \rceil \leq q \leq \lceil \frac{n+\sum\limits_{a=0}^{i+1}2k_{h-a}-2}{2} \rceil$, then $R_{\mathcal{A}}(G)= |\xi_{8+2i}| = 2t_{h-i-1}$ $(0\leq i \leq h-2)$.\smallskip

{\it Case 1.2.5.}
$t_0=2, p = 1$. Then $q \geq \lceil \frac{n}{2} \rceil$. Similarly to {\it Case 1.2.2}.
if $q = \lceil \frac{n}{2} \rceil$, then $R_{\mathcal{A}}(G) = |\xi_4| \in (2, 2t_h)$.
If $\lceil \frac{n+2}{2} \rceil \leq q \leq \lceil \frac{n+2k_h-2}{2}\rceil$, then $R_{\mathcal{A}}(G) = |\xi_5| = 2t_h$.
If $q= \lceil \frac{n+\sum\limits_{a=0}^{i}2k_{h-a}}{2} \rceil$, then $R_{\mathcal{A}}(G) = |\xi_{6+2i}| \in (2t_{h-i}, 2t_{h-i-1})$ $(0\leq i \leq h-1)$.
If $\lceil \frac{n+\sum\limits_{a=0}^{i}2k_{h-a}+2}{2} \rceil \leq q \leq \lceil \frac{n+\sum\limits_{a=0}^{i+1}2k_{h-a}-2}{2} \rceil$, then $R_{\mathcal{A}}(G)= |\xi_{7+2i}| = 2t_{h-i-1}$ $(0\leq i \leq h-2)$.\smallskip

{\it Case 1.2.6.}
$t_0=2, p = 0$. Then $q \geq \lceil \frac{n+2}{2} \rceil$. Similarly to {\it Case 1.2.3}, if $\lceil \frac{n+2}{2} \rceil \leq q \leq \lceil \frac{n+2k_h-2}{2}\rceil$, then $R_{\mathcal{A}}(G) = |\xi_4| = 2t_h$.
If $q= \lceil \frac{n+\sum\limits_{a=0}^{i}2k_{h-a}}{2} \rceil$, then $R_{\mathcal{A}}(G) = |\xi_{5+2i}| \in(2t_{h-i}, 2t_{h-i-1})$ $(0\leq i \leq h-1)$.
If $\lceil \frac{n+\sum\limits_{a=0}^{i}2k_{h-a}+2}{2} \rceil \leq q \leq \lceil \frac{n+\sum\limits_{a=0}^{i+1}2k_{h-a}-2}{2} \rceil$, then $R_{\mathcal{A}}(G)= |\xi_{6+2i}| = 2t_{h-i-1}$ $(0\leq i \leq h-2)$.\smallskip

{\it Case 2.} $n>t_0+p+q$. If $p+q \leq \lceil \frac{n-2t_0}{2} \rceil$, we get $\lceil\frac{n+1}{2}\rceil \leq 1+(n-t_0-p-q)$ and $R_{\mathcal{A}}(G)= |\xi_2|= 0$ by Proposition \ref{ES-eigenvalue}(ii) and (iii). For $p+q \geq \lceil \frac{n-2t_0+2}{2} \rceil$, we consider the the following cases that are similar to {\it Case 1}. If $p+q = \lceil \frac{n-2t_0+2}{2} \rceil$.
Then for $t_0=1, q=0$ or $t_0=3,q=4-p$ or $t_0=4,q =3-p$, $R_{\mathcal{A}}(G)= |\xi_2|= 0$; otherwise, for the other cases in conditions (i)-(v) of Proposition \ref{positive}, we get $R_{\mathcal{A}}(G) = |\xi_3| \in (0,1)$.

If $\lceil \frac{n-2t_0+4}{2} \rceil \leq p+q \leq \lceil \frac{n}{2} \rceil$, then $t_0 \geq 2$.
For $t_0=3,q=4-p$ or $t_0=4,q =3-p$, we get $R_{\mathcal{A}}(G)= |\xi_3|= 1$.
For the other cases in conditions (ii)-(v) of Proposition \ref{positive}, we have $R_{\mathcal{A}}(G)= |\xi_4|= 1$. If $p+q \geq \lceil \frac{n+2}{2} \rceil$,  we get  $\lceil\frac{n+1}{2}\rceil \geq 1+(n-t_0-p-q)+1+(t_0-1)+1$ and distinguish the following cases.

{\it Case 2.1.} If $t_0 =1$, $q=0$, then $p\geq 2$ and $q =0 \leq \lceil \frac{n-2}{2} \rceil$ which shows  $R_{\mathcal{A}}(G) = |\xi_3| = 2$ by Proposition \ref{ES-eigenvalue}(ii).
For $t_0=3,q=4-p$, then $p+q =4 \geq \lceil \frac{n+2}{2} \rceil$, and thus $n\leq 6 < t_0+p+q$, a contradiction.
For $t_0=4,q =3-p$, then $p+q =3 \geq \lceil \frac{n+2}{2} \rceil$, and so $n\leq 4 < t_0+p+q$, a contradiction.

{\it Case 2.2.} For the other cases in conditions (i)-(v) of Proposition \ref{positive}, if $t_0 \geq 5$, then $p+q =2\geq \lceil \frac{n+2}{2} \rceil$ and thus $n\leq 2 < t_0+p+q$; if $t_0 =4$, then
$\lceil \frac{n+2}{2} \rceil \leq p+q < 3$ and so $n<4 < t_0+p+q$; if $t_0 =3$, then
$\lceil \frac{n+2}{2} \rceil \leq p+q < 4$ and thus $n<6 < t_0+p+q$. In any case, we always have a contradiction. Hence, $t_0 \leq 2$.

{\it Case 2.2.1.} $t_0 =1, p\geq 2$. In this and next cases, we take use of the $\mathcal{A}$-spectrum of Proposition \ref{ES-eigenvalue}(ii).
If $q \leq \lceil \frac{n-2}{2} \rceil$, then $R_{\mathcal{A}}(G) = |\xi_4| = 2$.
If $q = \lceil \frac{n}{2} \rceil$, then $R_{\mathcal{A}}(G) = |\xi_5| \in(2, 2t_h)$.
If $\lceil \frac{n+2}{2} \rceil \leq q \leq \lceil \frac{n+2k_h-2}{2}\rceil$, then $R_{\mathcal{A}}(G) = |\xi_6| = 2t_h$.
If $q= \lceil \frac{n+\sum\limits_{a=0}^{i}2k_{h-a}}{2} \rceil$, then $R_{\mathcal{A}}(G) = |\xi_{7+2i}| \in(2t_{h-i}, 2t_{h-i-1})$ $(0\leq i \leq h-1)$.
If $\lceil \frac{n+\sum\limits_{a=0}^{i}2k_{h-a}+2}{2} \rceil \leq q \leq \lceil \frac{n+\sum\limits_{a=0}^{i+1}2k_{h-a}-2}{2} \rceil$, then $R_{\mathcal{A}}(G)= |\xi_{8+2i}| = 2t_{h-i-1}$ $(0\leq i \leq h-2)$.

{\it Case 2.2.2.}
$t_0 =1, p = 1$. Then $q \geq \lceil \frac{n}{2} \rceil$.
If $q = \lceil \frac{n}{2} \rceil$, then $R_{\mathcal{A}}(G) = |\xi_4| \in (2, 2t_h)$.
If $\lceil \frac{n+2}{2} \rceil \leq q \leq \lceil \frac{n+2k_h-2}{2}\rceil$, then $R_{\mathcal{A}}(G) = |\xi_5| = 2t_h$.
If $q= \lceil \frac{n+\sum\limits_{a=0}^{i}2k_{h-a}}{2} \rceil$, then $R_{\mathcal{A}}(G) = |\xi_{6+2i}| \in (2t_{h-i}, 2t_{h-i-1})$ $(0\leq i \leq h-1)$.
If $\lceil \frac{n+\sum\limits_{a=0}^{i}2k_{h-a}+2}{2} \rceil \leq q \leq \lceil \frac{n+\sum\limits_{a=0}^{i+1}2k_{h-a}-2}{2} \rceil$, then $R_{\mathcal{A}}(G)= |\xi_{7+2i}| = 2t_{h-i-1}$ $(0\leq i \leq h-2)$.

{\it Case 2.2.3.}
$t_0 =1, p = 0$. In this  case, we employ the $\mathcal{A}$-spectrum of Proposition \ref{ES-eigenvalue} (ii). Then $q \geq \lceil \frac{n+2}{2} \rceil$.
If $\lceil \frac{n+2}{2} \rceil \leq q \leq \lceil \frac{n+2k_h-2}{2}\rceil$, then $R_{\mathcal{A}}(G) = |\xi_4| = 2t_h$.
If $q= \lceil \frac{n+\sum\limits_{a=0}^{i}2k_{h-a}}{2} \rceil$, then $R_{\mathcal{A}}(G) = |\xi_{5+2i}| \in (2t_{h-i}, 2t_{h-i-1})$ $(0\leq i \leq h-1)$.
If $\lceil \frac{n+\sum\limits_{a=0}^{i}2k_{h-a}+2}{2} \rceil \leq q \leq \lceil \frac{n+\sum\limits_{a=0}^{i+1}2k_{h-a}-2}{2} \rceil$, then $R_{\mathcal{A}}(G)= |\xi_{6+2i}| = 2t_{h-i-1}$ $(0\leq i \leq h-2)$.

{\it Case 2.2.4.}
$t_0 =2, p\geq 2$. Similarly, if $q \leq \lceil \frac{n-2}{2} \rceil$, then $R_{\mathcal{A}}(G) = |\xi_5| = 2$.
If $q = \lceil \frac{n}{2} \rceil$, then $R_{\mathcal{A}}(G) = |\xi_6| \in (2, 2t_h)$.
If $\lceil \frac{n+2}{2} \rceil \leq q \leq \lceil \frac{n+2k_h-2}{2}\rceil$, then $R_{\mathcal{A}}(G) = |\xi_7| = 2t_h$.
If $q= \lceil \frac{n+\sum\limits_{a=0}^{i}2k_{h-a}}{2} \rceil$, then $R_{\mathcal{A}}(G) = |\xi_{8+2i}| \in (2t_{h-i}, 2t_{h-i-1})$ $(0\leq i \leq h-1)$.
If $\lceil \frac{n+\sum\limits_{a=0}^{i}2k_{h-a}+2}{2} \rceil \leq q \leq \lceil \frac{n+\sum\limits_{a=0}^{i+1}2k_{h-a}-2}{2} \rceil$, then $R_{\mathcal{A}}(G)= |\xi_{9+2i}| = 2t_{h-i-1}$ $(0\leq i \leq h-2)$.

{\it Case 2.2.5.}
$t_0 =2, p = 1$. Then $q \geq \lceil \frac{n}{2} \rceil$. Similarly, if $q = \lceil \frac{n}{2} \rceil$, then $R_{\mathcal{A}}(G) = |\xi_5| \in (2, 2t_h)$.
If $\lceil \frac{n+2}{2} \rceil \leq q \leq \lceil \frac{n+2k_h-2}{2}\rceil$, then $R_{\mathcal{A}}(G) = |\xi_6| = 2t_h$.
If $q= \lceil \frac{n+\sum\limits_{a=0}^{i}2k_{h-a}}{2} \rceil$, then $R_{\mathcal{A}}(G) = |\xi_{7+2i}| \in (2t_{h-i}, 2t_{h-i-1})$ $(0\leq i \leq h-1)$.
If $\lceil \frac{n+\sum\limits_{a=0}^{i}2k_{h-a}+2}{2} \rceil \leq q \leq \lceil \frac{n+\sum\limits_{a=0}^{i+1}2k_{h-a}-2}{2} \rceil$, then $R_{\mathcal{A}}(G)= |\xi_{8+2i}| = 2t_{h-i-1}$ $(0\leq i \leq h-2)$.

{\it Case 2.2.6.}
$t_0 =2, p = 0$. Then $q \geq \lceil \frac{n+2}{2} \rceil$. Similarly,
 if $\lceil \frac{n+2}{2} \rceil \leq q \leq \lceil \frac{n+2k_h-2}{2}\rceil$, then $R_{\mathcal{A}}(G) = |\xi_5| = 2t_h$.
If $q= \lceil \frac{n+\sum\limits_{a=0}^{i}2k_{h-a}}{2} \rceil$, then $R_{\mathcal{A}}(G) = |\xi_{6+2i}| \in (2t_{h-i}, 2t_{h-i-1})$ $(0\leq i \leq h-1)$.
If $\lceil \frac{n+\sum\limits_{a=0}^{i}2k_{h-a}+2}{2} \rceil \leq q \leq \lceil \frac{n+\sum\limits_{a=0}^{i+1}2k_{h-a}-2}{2} \rceil$, then $R_{\mathcal{A}}(G)= |\xi_{7+2i}| = 2t_{h-i-1}$ $(0\leq i \leq h-2)$.

This finishes the proof.

\section{Remarks}

As Haemers pointed out \cite{Haem1}, the mixed extension of graphs is a special case of the so called {\it generalized composition} (see \cite{sch-join} for details). However, Haemers's definition is more convenient and powerful. It is quite helpful for the classifications of graphs and further for identifying which graphs are determined by the spectra. On reflection, we propose the following problems, the first one of which is a natural step.

\begin{prob}
Which graphs with exactly one positive $\mathcal{A}$-eigenvalue are determined by their $\mathcal{A}$-spectra?
\end{prob}

\begin{prob}
Determine the graphs with all but two $\mathcal{A}$-eigenvalues equal to $-2$ and $-1$.
\end{prob}

For the anti-adjacency matrix, we expect more general results about the  {\rm HL}-index $R_{\mathcal{A}}(G)$ of a graph $G$.

\begin{prob}
Investigate the HL-index w.r.t. the anti-adjacency matrices of graphs.
\end{prob}

Another interesting problem is the {\it nullity} $\eta_A(G)$ of a graph $G$, which is defined to be the multiplicity of zero as an eigenvalue of the adjacency matrix.  Similarly, we consider the nullity $\eta_{\mathcal{A}}(G)$  of the anti-adjacency matrix. The nullity of a graph is important in mathematics, since it is related to the singularity of (anti-)adjacency matrix.  Note that the anti-adjacency matrices of graphs seem to be usually more sparse than adjacency matrices. Hence, the nullity of anti-adjacency matrix may be larger than that of adjacency matrix. For example, $\eta_A(P_{2k+1}) =1$ and $\eta_A(P_{2k}) = 0$ for $k \geq 1$; while $\eta_{\mathcal{A}}(P_{2k+1}) = 2k-3$ and $\eta_{\mathcal{A}}(P_{2k}) = 2k-4$ for $k \geq 3$ (see \cite[Lemma 2.1]{wang-lu-DM})

\begin{prob}
For the anti-adjacency matrix, give lower and upper bounds of nullity involving graph parameters, and characterize the extreme graphs.
\end{prob}

On the other hand, it is well-know that the HL-index and the nullity of graphs are of great interest in chemistry.  As a comparison in \cite{wang-li-CILS}, it seems that the spectral radius of adjacency matrix of a graph is closely related to the  chemical properties of octane isomers, while the spectral radius of anti-adjacency matrix of graphs may be more efficient for the benzenoid hydrocarbons. In the end,
we pose the last one problem to finish this paper.

\begin{prob}
With respect to the adjacency and anti-adjacency matrices of graphs, compare the HL-index and nullity of graphs
and study their applications in the chemistry.
\end{prob}

\section*{Acknowledgement}

The first author and the last two authors would like to thank to the Professor Willem H. Haemers for sharing his pearls of wisdom with them
during the course of this research and they are grateful for his comments on an earlier version of the manuscript.

We are also grateful to two anonymous peers for their many helpful comments and suggestions, which have considerably improved the presentation of the first draft.

\small{

}
\end{document}